\documentclass[10pt,a4paper,reqno]{amsart}
\makeatletter

\newcommand{\Rmnum}[1]{\expandafter\@slowromancap\romannumeral #1@}
\makeatother
\usepackage{jiuya}

\begin{document}
\title{Generalized Erd\H{o}s-Tur\'an inequalities and \\stability of energy minimizers}
\author{Ruiwen Shu and Jiuya Wang}
\newcommand{\Addresses}{{
	    \footnotesize		
		Ruiwen Shu, \textsc{Mathematical Institute, University of Oxford, Oxford OX2 6GG, UK
		}\par\nopagebreak
		\textit{E-mail address}: \texttt{shu@maths.ox.ac.uk}	
		
		\bigskip
		\footnotesize		
		Jiuya Wang, \textsc{Department of Mathematics, University of Georgia, Boyd Graduate Studies Research Center, Athens, GA 30601, USA
		}\par\nopagebreak
		\textit{E-mail address}: \texttt{jiuya.wang@uga.edu}	
	}}
\maketitle	
\vspace{-0.5cm}
	\begin{abstract}
		The classical Erd\H{o}s-Tur\'an inequality on the distribution of roots for complex polynomials can be equivalently stated in a potential theoretic formulation, that is, if the logarithmic potential generated by a probability measure on the unit circle is close to $0$, then this probability measure is close to the uniform distribution. We generalize this classical inequality from $d=1$ to higher dimensions $d>1$ with the class of Riesz potentials which includes the logarithmic potential as a special case. In order to quantify how close a probability measure is to the uniform distribution in a general space, we use Wasserstein-infinity distance as a canonical extension of the concept of discrepancy. Then we give a compact description of this distance. Then for every dimension $d$, we prove inequalities bounding the Wasserstein-infinity distance between a probability measure $\rho$ and the uniform distribution by the $L^p$-norm of the Riesz potentials generated by $\rho$. Our inequalities are proven to be sharp up to the constants for singular Riesz potentials. Our results indicate that the phenomenon discovered by Erd\H{o}s and Tur\'an about polynomials is much more universal than it seems. Finally we apply these inequalities to prove stability theorems for energy minimizers, which provides a complementary perspective on the recent construction of energy minimizers with clustering behavior.
	\end{abstract}
\bf Keywords. \normalfont Erd\H{o}s-Tur\'an inequality, energy minimization, Wasserstein distance, discrepancy, potential theory, stability
\pagenumbering{arabic}

\section{Introduction}
\subsection{Main Theorem}
In 1950, Erd\H{o}s and Tur\'an \cite{ET} prove a classical inequality on the distribution of roots of a complex polynomial $f(z)\in \C[z]$. The inequality characterizes the phenomenon that if $f(z)$ attains small value on the unit circle, then the angular distribution of the roots of $f(z)$ is close to equidistribution. Let $f(z) = \sum_{k=0}^{n} a_k z^k  \in \C[z]$ be a polynomial where $a_0a_n\neq 0$ and denote its roots by $r_j e^{2\pi i\theta_j}$ for $1\le j\le n$ with $\theta_j \in \bT =\mathbb{R}/\mathbb{Z}$. For $ \alpha\le \beta < \alpha+1$, we write $N_f(\alpha, \beta)$ to be the number of roots with $\theta_j \in [\alpha, \beta]$ when considered as a subset in $\bT$. We define for a polynomial $f$ that
\begin{equation}
	\mathcal{D}[f]:= \max_{ \alpha\le \beta< \alpha+1} \frac{N_f(\alpha, \beta)}{n}-(\beta-\alpha), \quad 	\mathcal{H}[f]:=\frac{1}{n} \log \frac{\max_{|z|=1}|f(z)|}{\sqrt{|a_0a_n|}},
\end{equation}
where $\cD[f]$ is the \emph{discrepancy} of $f$ and measures the deviation of the angle distribution of roots away from the uniform distribution on the unit circle, and $\cH[f]$ is the \emph{height} of $f$. Then the Erd\H{o}s-Tur\'an inequality states that there exists an absolute constant $C$ such that 
\begin{equation}\label{eqn:ET}
	\cD[f] \le C \cH[f]^{1/2}.
\end{equation}

In a recent work \cite{SW21}, the authors prove that the optimal constant in \eqref{eqn:ET} is $C = \sqrt{2}$. One of the main ideas in \cite{SW21} is to consider inequality \eqref{eqn:ET} for all probability measures in a potential theoretic formulation. An observation due to Schur \cite{Schur} shows that to prove \eqref{eqn:ET} it suffices to consider polynomials $f$ with all roots on the unit circle. Therefore by extending $\cD$ and $\cH$ to all probability distributions as
%
\begin{equation}\label{eqn:D-H-def}
	\cD[\rho] = \sup_{I\subset \bT} \int_{I} (\rho-1) \rd{x}, \quad \cH[\rho] = - \ess\inf (W_{\log}*\rho),
\end{equation}
where the supreme is taken over closed intervals $I$ of $\bT$ and $W_{\log}(x) = -\log |2\sin \pi x|$ for $x \in \bT $, the authors turn \eqref{eqn:ET} from a discrete question to a continuous one. 
%
%
In potential theory, given an interaction potential $W$, we denote $V_{\rho}:= W*\rho$ to be the total potential generated by $\rho$ under the potential $W$. Then combining Schur's observation, Erd\H{o}s-Tur\'an inequality can be equivalently stated in the following potential theoretic formulation 
%
\begin{equation}\label{eqn:ET-potential}
\cD[\rho] \le \sqrt{2} \cdot \| (V[1_{\mathbb{T}}]-V[\rho])_+ \|_{L^\infty}^{1/2} = \sqrt{2} \cdot \| (-W*\rho)_{+} \|^{1/2}_{L^\infty},
\end{equation}
where the uniform distribution $1_{\bT}$ is the unique probability distribution such that $V[1_{\bT}] = 0$, see \cite[Theorem $3.6$]{SW21} for an argument. Under this formulation, the inequality states that if the total potential generated by $\rho$ is close to the total potential generated by $1_{\bT}$, then $\rho$ is also close to $1_{\bT}$. 

The original inequality \eqref{eqn:ET} in \cite{ET} is stated in terms of polynomials since the motivation of Erd\H{o}s and Tur\'an lies in number theory and complex analysis, however using the polynomial formulation will force us to take the logarithmic potential. Via stating this inequality in terms of the potential theoretic formulation, we eliminate this restriction and fit this question into a much more general framework. 

In this paper, we will generalize the Erd\H{o}s-Tur\'an inequality \eqref{eqn:ET-potential} in the following directions:
\begin{itemize}
	\item 
	\textbf{Dimension}: instead of only considering probability measures over $\bT$, we will consider $\bT^d$ for general $d$.
	\item
	\textbf{Potential}: instead of only using the logarithmic potential $W_{\log} = -\log|2\sin \pi x|$, we will consider the class of periodized Riesz potentials. For $s<d$, the periodized Riesz potential $W_s$ is defined by
\begin{equation}\label{W}
	\hat{W}_s(\bk)= |\bk|^{-d+s},\,\forall 0\ne \bk \in\mathbb{Z}^d,\quad \hat{W}(0)=0.
\end{equation}
In fact $W_{\log} = -\log|2 \sin \pi x|$ is the special case of $W_s$ with $d=1,\,s=0$.
	\item
	\textbf{Height}: instead of measuring $W*\rho$ by $\cH[\rho]$ which is close to an $L^{\infty}$-norm of $W*\rho$, we measure $W*\rho$ by its $L^p$-norm. 
\end{itemize}

A critical issue in generalizing this inequality to higher dimension is how to generalize the notation of \emph{discrepancy} in $d=1$. 
We adopt the view in \cite[Proposition 2]{Wass} that the discrepancy $\cD[\rho]$ for $d=1$ is equivalent to the Wasserstein-infinity distance between $\rho$ and the uniform distribution $1_{\bT}$, that is,
\begin{equation}\label{D1}
\frac{1}{2}\cD[\rho]  = d_\infty(\rho,1_{\bT}),  \text{  for } \rho\in \cM(\mathbb{T}).
\end{equation}
We recall that for a locally compact topological space $X$ with a distance function $\dist(\cdot, \cdot)$, the Wasserstein infinity distance between two probability measures $\rho_1,\rho_2$ on $X$ is defined as
\begin{equation}\label{dinfty}
	d_\infty(\rho_1,\rho_2): = \inf_{\mu\in \Pi(\rho_1,\rho_2)} \sup_{(x,y)\in \supp\mu} \dist(x,y),
\end{equation}
where $\Pi(\rho_1,\rho_2)$ is the set of transport plans from $\rho_1$ to $\rho_2$, i.e., those probability measures $\mu(x,y)$ on $X\times X$ with $\int_X \mu(x,y)\rd{y} = \rho_1(x)$ and $\int_X \mu(x,y)\rd{x} = \rho_2(y)$. Therefore we can generalize the notion of discrepancy in a canonical way as long as the underlying space $X$ is locally compact and equipped with a metric.

\begin{theorem}\label{thm1}
	Let $s<d$, $1\le p < \infty$ and $\rho\in \cM(\bT^d)$. 
	\begin{enumerate}
		\item[\textnormal{(i)}]  If $1\le p \le d$, then 
		\begin{equation}\label{thm1_1}
			d_\infty(\rho,1) \lesssim \|W_s*\rho\|_{L^p}^\gamma,\quad \gamma = \frac{1}{d+d/p-s}.
		\end{equation}
		\item[\textnormal{(ii)}] If $d=1$, $p>1$, then
		\begin{enumerate}
			\item If $s < \frac{1}{p}$, then \eqref{thm1_1} also holds.
			\item If $s = \frac{1}{p}$, then
			\begin{equation}\label{thm1_2}
				d_\infty(\rho,1) (1+|\log d_\infty(\rho,1)|)^{-1+1/p} \lesssim  \|W_s*\rho\|_{L^p}.
			\end{equation}
			\item If $ \frac{1}{p} < s < 1$, then
			\begin{equation}\label{thm1_3}
				d_\infty(\rho,1)  \lesssim \|W_s*\rho\|_{L^p}.
			\end{equation}
		\end{enumerate}
	    	\item[\textnormal{(iii)}]   \textnormal{(i)} and \textnormal{(ii)} are sharp up to the constants when $0<s<d$ with the exception $d = 1$ and $s = 1/p$.
	\end{enumerate}
	Here the implied constants may depend on $d$, $p$, and $s$. 
\end{theorem}
\begin{remark}
	For $d=1$, \cite{Wass} shows that $d_\infty(\rho,1) = \|\rho-1\|_{\dot{W}^{-1,\infty}}$ is a negative Sobolev norm, and the same is true for $\|W_s*\rho\|_{L^2} = \|\rho-1\|_{\dot{H}^{-1+s}}$. Therefore,  for the case $p=2$, one can view \eqref{thm1_3} as a Sobolev embedding. However, in the case $s<1/p = 1/2$, the inequality \eqref{thm1_1} cannot be viewed in this way because the homogeneous degrees on its two sides are different. In other words, although $\|\rho-1\|_{\dot{W}^{-1,\infty}}^{3/2-s} \le C\|\rho-1\|_{\dot{H}^{-1+s}}$ is true for $\rho\in\cM$, one does not expect the same to be true if $\rho-1$ is replaced by a general mean-zero signed measure $\mu$.
\end{remark}
\begin{remark}
	To see the optimality of the scalings in Theorem \ref{thm1} for $s\le 0$, one would need a better description of $W_s$ near $0$, similar to item 2 of Lemma \ref{lem_period} for the case $0<s<d$. This is left as future work.
\end{remark}

We now compare Theorem \ref{thm1} with previous results in the literature. Our theorem includes the case $d = 1$, $p=1$ with $W_{\log}$ as a special case, see previous work in \cite{Mignotte,Sound,aim}. In $d=1$, another alternative height being used before is energy see \cite{kleiner1964equilibrium,huesing}, and over $\bT$ we will show in Section \ref{sec_eng} that the energy is essentially a $L^2$-norm with certain Riesz potential. In higher dimensions, our result is fundamentally different from previous results \cite{sjogrenhighdim,gotz,kleiner1964equilibrium,Wagner} in that we use different generalization of discrepancy. In these works, a discrepancy in the form of supreme of $\int_S (\rho-1) \rd{x}$ over certain test sets $S$ was used, however, these results are restricted in the sense that the inequality has a dependency on extra parameters from the choice of $S$. We also mention \cite{Wagner} on bounding the discrepancy in Wasserstein-1 distance and the more recent work \cite{steinerwass} on bounding the discrepancy in other Wasserstein distances in $d=1$. 

We close this section by the following remark: the fact that we are able to prove this inequality in this generality shows that the phenomenon that what Erd\H{o}s-Tur\'an discovered about polynomials actually holds in a much more universal way.
\subsection{Application on Stability of Energy Minimizers}\label{sec_eng}
In this section, we will give an application of Theorem \ref{thm1} on study of energy minimization. 

In potential theory, the probability measure(s) $\rho$ that minimizes the potential energy 
\begin{equation}
	\cE_W[\rho]:= \frac{1}{2} \int_X (W*\rho)\cdot \rho \rd{x},
\end{equation}
on a certain space $X$ with interaction potential $W$ are called \emph{energy minimizers}. Energy minimizers for the pairwise interaction energy $\cE_W$ on $\mathbb{R}^d$ have been studied extensively, in terms of existence, uniqueness, and properties \cite{Ca13,Ca13_2,CCP15,SST15,Lop19,ShuCar,BCT,BKSUB,CDM,CFP17,KSUB,ST}. Following these results, a natural question is the \emph{stability} of energy minimizer. That is to say, in case there is a unique energy minimizer $\rho_\infty$, whether it is possible to estimate the distance between $\rho\in\cM$ and $\rho_\infty$ in terms of $\cE_W[\rho]-\cE_W[\rho_\infty]$.

A crucial observation we made is a connection between the energy $\cE_W$ and $L^2$-norm of the generated potential using Fourier transform.
Over $\bT^d$, the potential energy with the interaction potential $W$ is
\begin{equation}\label{eqn:L2-energy}
	\cE_W[\rho]:= \frac{1}{2} \int_{\bT^d} (W*\rho)\cdot \rho \rd{\bx} = \frac{1}{2}\sum_{\bk\in\mathbb{Z}^d} \hat{W}(\bk)|\rho(\bk)|^2,
\end{equation}
where the last equality is justified in \cite[Appendix $2$]{SW21} when $W$ is nice, see (\textbf{H1}) to (\textbf{H3}) below. From the Fourier side $\cE_W[\rho]$ is actually equivalent to $\| \mathcal{W}* \rho \|^{2}_{L^2}$ for another potential function $\mathcal{W}$ when $\hat{\mathcal{W}}(\bk)^2 = \hat{W}(\bk)$, by applying Plancherel identity. This justifies our generalization to Riesz potentials, since even one starts with a logarithmic potential, its energy is exactly the $\| \mathcal{W}*\rho \|_{L^2}^2$ for some Riesz potential, moreover the class of Riesz potentials is closed under taking square-root on the Fourier coefficients. 


On the other hand, if the interaction potential satisfies $\hat{W}(\bk)>0$ for $\bk\neq0$, then it is clear from the expression on the Fourier side \eqref{eqn:L2-energy} that the uniform distribution $\rho_\infty=1$ is the unique energy minimizer with the minimal energy $\cE_W[\rho_\infty]=0$.

We now apply Theorem \ref{thm1} to prove the stability of certain energy minimizers in the $d_\infty$ sense based on the observations above. Let $W:\mathbb{T}^d\rightarrow (-\infty,\infty]$ be an interaction potential function satisfying the following assumptions:
\begin{itemize}
	\item ({\bf H1}) $W$  is $L^1$, even, lower-semicontinuous and bounded from below.
	\item ({\bf H2}) $\hat{W}(0)=0$, and $\hat{W}(\bk)>0$ for any $\bk\in\mathbb{Z}^d\backslash \{0\}$.
	\item ({\bf H3}) For some $C_1>0$, there holds $\frac{1}{|B(\bx;r)|}\int_{B(\bx;r)}(C_1+W(\by))\rd{\by} \le C(C_1+W(\bx))$ for any $r>0$.
\end{itemize}
When $W$ is a periodized Riesz potential $W_s$,  Lemma \ref{lem_period} shows that $W_s$ satisfies ({\bf H1})-({\bf H3}) for any $s<d$, and we denote the corresponding energy as $\cE_s$.



Our goal is to get a stability estimate by controlling $d_\infty(\rho,1)$ in terms of $\cE_W[\rho]$. We apply \eqref{eqn:L2-energy} to Riesz potentials
\begin{equation}
	\cE_s[\rho] = \frac{1}{2}\sum_{\bk\in\mathbb{Z}^d} |\bk|^{-d+s}|\rho(\bk)|^2 = \frac{1}{2}\sum_{\bk\in\mathbb{Z}^d} \big||\bk|^{(-d+s)/2}\rho(\bk)\big|^2 = \frac{1}{2}\|W_{s'}*\rho\|_{L^2}^2,
\end{equation}
with $s'=\frac{d+s}{2}$. Also notice that if a potential $W$ satisfies $\hat{W}(0)=0$ and $\hat{W}(\bk)\ge c\hat{W}_s(\bk)$, then $\cE_W[\rho] \ge c \cE_s[\rho]$. Therefore, applying Theorem \ref{thm1} with $p=2$ and $s$ replaced by $s'$, we directly get the following result on the stability of the uniform distribution as an energy minimizer, with optimal scaling up to the possible logarithmic factor.
\begin{theorem}\label{thm3}
	Let $s<d$ and $\rho\in \cM$. Let $W$ be an interaction potential satisfying ({\bf H1})-({\bf H3}) with a quantitative lower bound
	\begin{equation}
		\hat{W}(\bk) \ge c|\bk|^{-d+s},\,\forall 0\ne \bk \in\mathbb{Z}^d,\quad \hat{W}(0)=0.
	\end{equation}
	Then the associated energy $\cE_W[\rho] = \frac{1}{2}\int_{\mathbb{T}^d}(W*\rho)\rho\rd{\bx}$ satisfies:
	\begin{enumerate}
		\item[\textnormal{(i)}]  If $d\ge 2$, then
		\begin{equation}\label{thm3_1}
			d_\infty(\rho,1) \lesssim  \cE_W[\rho]^\gamma,\quad \gamma = \frac{1}{2d-s}.
		\end{equation}
		\item[\textnormal{(ii)}] If $d=1$, then
		\begin{enumerate}
			\item If $s < 0$, then \eqref{thm3_1} also holds.
			\item If $s = 0$, then
			\begin{equation}\label{thm3_2}
				d_\infty(\rho,1)(1+|\log d_\infty(\rho,1)|)^{-1/2} \lesssim  \cE_W[\rho]^{1/2}.
			\end{equation}
			\item If $ 0<s<1$, then
			\begin{equation}
				d_\infty(\rho,1)  \lesssim  \cE_W[\rho]^{1/2}.
			\end{equation}
		\end{enumerate}
	    	\item[\textnormal{(iii)}]   \textnormal{(i)} and \textnormal{(ii)} are sharp up to the constants when $0<s<d$ with the exception $d = 1$ and $s = 1/p$.
	\end{enumerate}
\end{theorem}

\begin{remark}
	The counterpart of $\cE_s$ on $\mathbb{R}$ in the case $s=0$ was studied by \cite{CFP}, in which the wellposedness of the Wasserstein-2 gradient flow associated to this energy is proved. In the presence of a quadratic attractive potential, \cite{CFP} also proves the exponential convergence to the energy minimizer. This result implies a stability result of the form $d_2(\rho,\rho_\infty) \le C(\cE[\rho]-\cE[\rho_\infty])^{1/2}$, where $\rho_\infty$ denotes the unique energy minimizer, and $d_2$ denotes the Wasserstein-2 distance. Since $d_2(\rho_1,\rho_2) \le d_\infty(\rho_1,\rho_2)$ in any underlying space, our result \eqref{thm3_2} takes a stronger form than the stability result implied by \cite{CFP}, up to the logarithmic factor.
\end{remark}

Finally we give a result on the stability of energy minimizers with respect to the perturbation on the potential $W$. 
\begin{theorem}\label{thm4}
	Let $W$ satisfy the assumption of Theorem \ref{thm3}, and $\tilde{W}$ is a perturbation of $W$, satisfying $\|W-\tilde{W}\|_{L^\infty} < \infty$. Then any minimizer $\rho$ of the interaction energy $\cE_{\tilde{W}}$ in $\cM$ satisfies the same conclusions as in Theorem \ref{thm3} with $\cE_W[\rho]$ replaced by $\|W-\tilde{W}\|_{L^\infty}$.
\end{theorem}
We remark that the existence of minimizers for $\cE_{\tilde{W}}$ can be guaranteed as long as $\tilde{W}$ satisfies ({\bf H1}), see Proposition 3.8 of \cite{SW21} for a treatment on $\mathbb{T}$ (which can be easily generalized to $\mathbb{T}^d$).
\begin{proof}
	As a minimizer of $\cE_{\tilde{W}}$, $\rho$ satisfies
	\begin{equation}
		\cE_{\tilde{W}}[\rho] \le \cE_{\tilde{W}}[1] = \cE_W[1] + \frac{1}{2}\int\int (\tilde{W}(\bx-\by)-W(\bx-\by))\rd{\by}\rd{\bx} \le \frac{1}{2}\|W-\tilde{W}\|_{L^\infty}. 
	\end{equation}
	Therefore
	\begin{equation}\begin{split}
			\cE_W[\rho]  = & \cE_{\tilde{W}}[\rho] + \frac{1}{2}\int\int (W(\bx-\by)-\tilde{W}(\bx-\by))\rho(\by)\rd{\by}\rho(\bx)\rd{\bx} \\ \le & \frac{1}{2}\|W-\tilde{W}\|_{L^\infty} + \frac{1}{2}\|W-\tilde{W}\|_{L^\infty} = \|W-\tilde{W}\|_{L^\infty}.
	\end{split} \end{equation}
	Then we get the conclusion by applying Theorem \ref{thm3} to $\rho$.
\end{proof}
Although the proof is simple, Theorem \ref{thm4} actually gives an interesting perspective on energy minimizers. Indeed, even a small $L^\infty$ perturbation on $W$ may destroy the positivity condition ({\bf H2}) and result in complicated energy minimizer(s). However the energy minimizer(s) under the perturbed potential will stay close to the original one in the sense of Wasserstein-infinity distance. In fact, we will consider the following example, which uses a similar idea as Sections 7 and 8  of \cite{ShuCar}. We take $W=W_s$ with $s<0$ and 
\begin{equation}
	\tilde{W}_\epsilon(\bx) = W(\bx) - c_0\epsilon^{-s} \psi\big(\frac{\bx}{\epsilon}\big),\quad \epsilon>0,\quad \bx\in \mathbb{T}^d=[-1/2,1/2)^d.
\end{equation}
where $\psi$ is a fixed compactly supported mollifier (i.e., $\psi$ is nonnegative, radial, and $\int_{\mathbb{R}^d}\psi\rd{\bx} = 1$), and $c_0>0$ is a constant to be chosen. Then one has $\|W-\tilde{W}_\epsilon\|_{L^\infty}\rightarrow 0$ as $\epsilon\rightarrow 0+$. On the other hand,
\begin{equation}\label{Wc0}
	\cF[\tilde{W}_\epsilon](\bk) = |\bk|^{-d+s} - c_0\epsilon^{-s+d} \hat{\psi}(\epsilon \bk),\quad \forall 0\ne \bk \in \mathbb{Z}^d.
\end{equation}
where $\hat{\psi}$ is the Fourier transform of $\psi$ on $\mathbb{R}^d$. Clearly $\hat{\psi}(\xi)$ is real and bounded from below by $1/2$ in some ball $\xi\in B(0;R)$, $R>0$. Therefore, for any $\bk$ with $R/(2\epsilon)\le |\bk| \le R/\epsilon$, we have
\begin{equation}
	\cF[\tilde{W}_\epsilon](\bk) \le \big(\frac{R}{2\epsilon}\big)^{-d+s} - c_0\epsilon^{-s+d} \cdot \frac{1}{2} = -\Big(\frac{1}{2}c_0-\big(\frac{R}{2}\big)^{-d+s}\Big)\epsilon^{-s+d} < -c\epsilon^{-s+d},
\end{equation}
if $c_0>2(\frac{R}{2})^{-d+s}$. In this case, $\cF[\tilde{W}_\epsilon]$ always attains negative values for any $\epsilon>0$, which implies that the uniform distribution $1$ is \emph{not} a minimizer for $\cE_{\tilde{W}_\epsilon}$. Indeed, Lemma 7.3 and Remark 7.4 of \cite{ShuCar} suggest that $\cE_{\tilde{W}_\epsilon}$ is likely to have minimizers consisting of clusters of radius at most $O(\epsilon)$. See Figure \ref{fig1} for numerical evidence for this phenomenon. 
\begin{figure}
	\begin{center}
		\includegraphics[width=.48\linewidth]{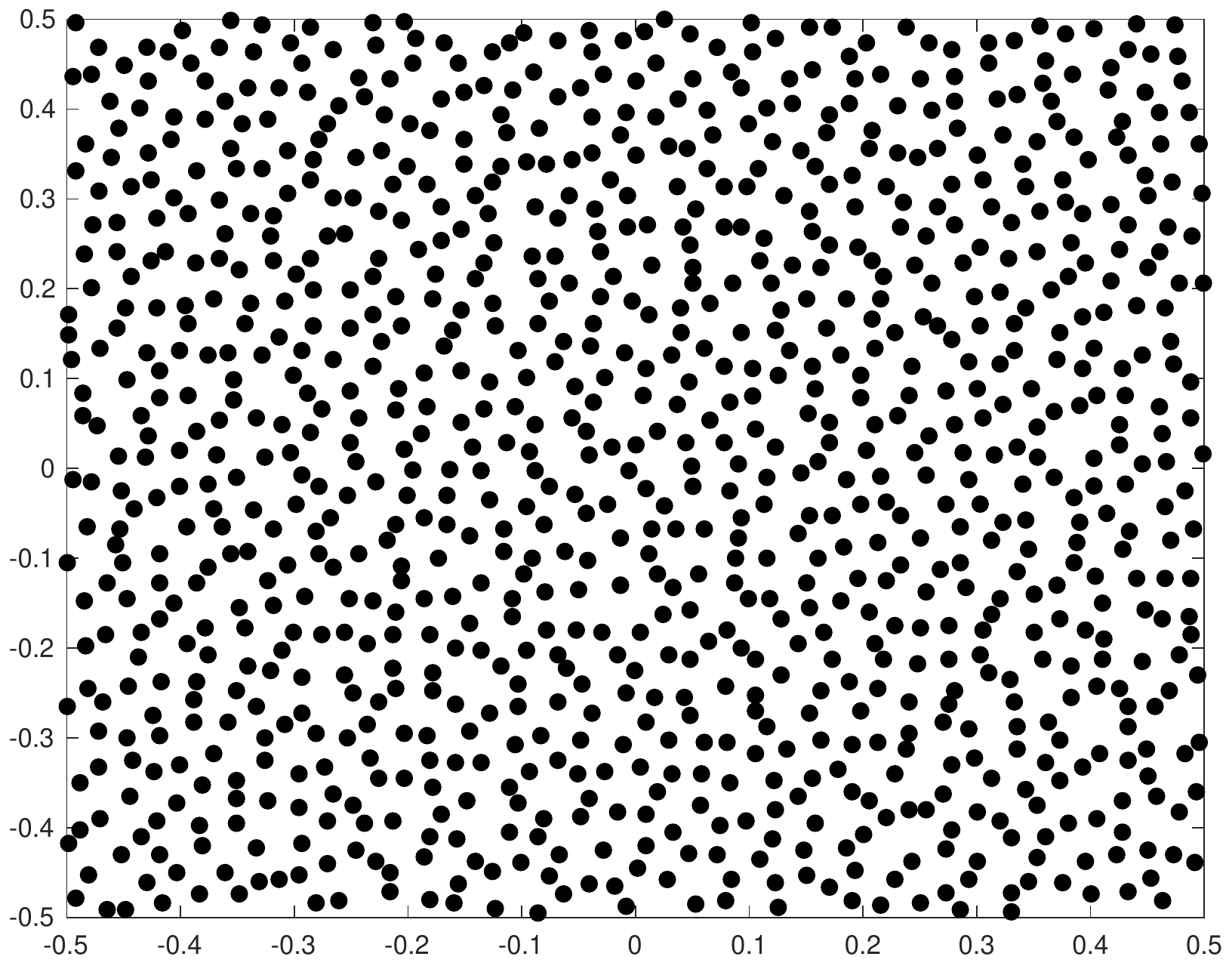}
		\includegraphics[width=.48\linewidth]{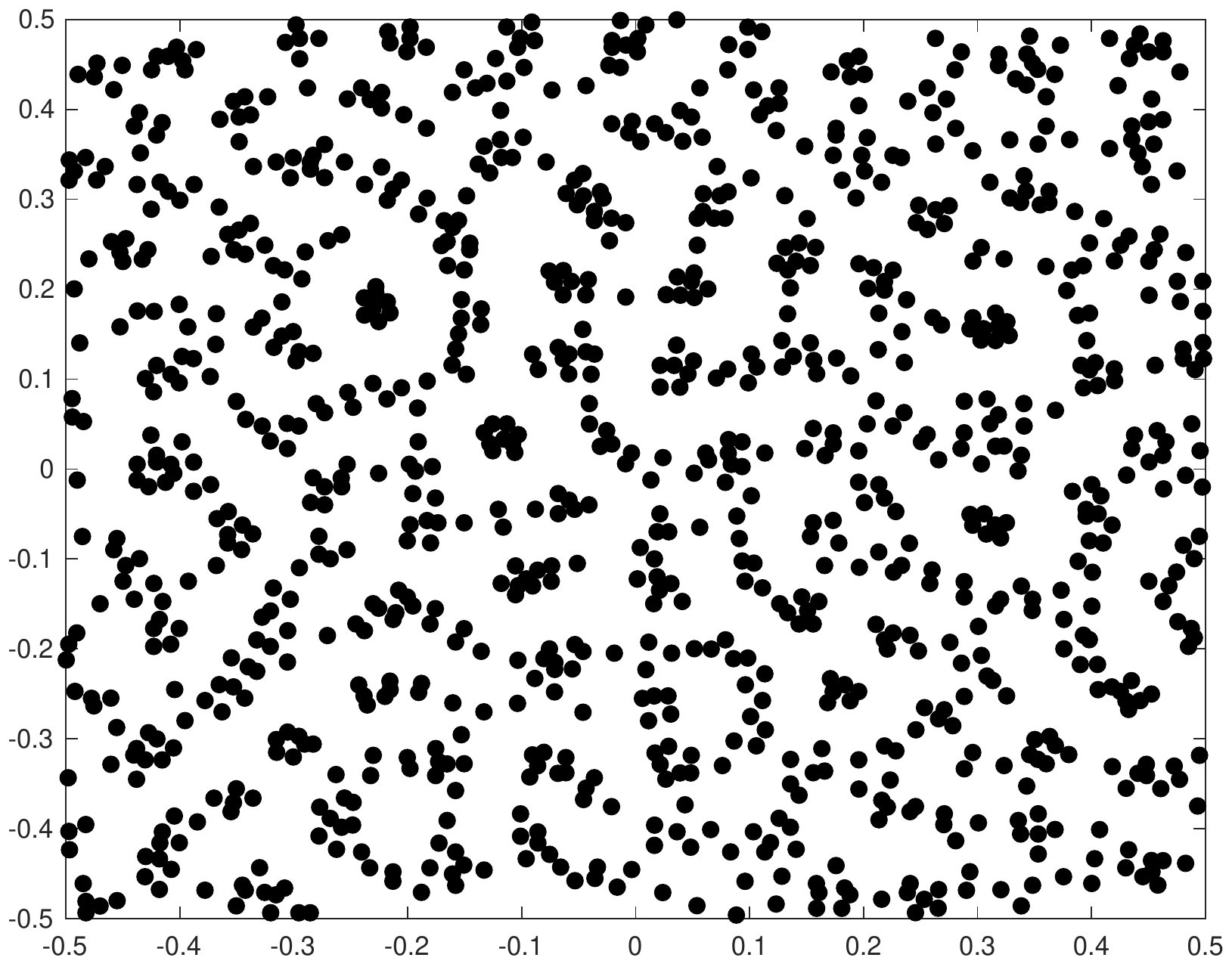}
		\includegraphics[width=.48\linewidth]{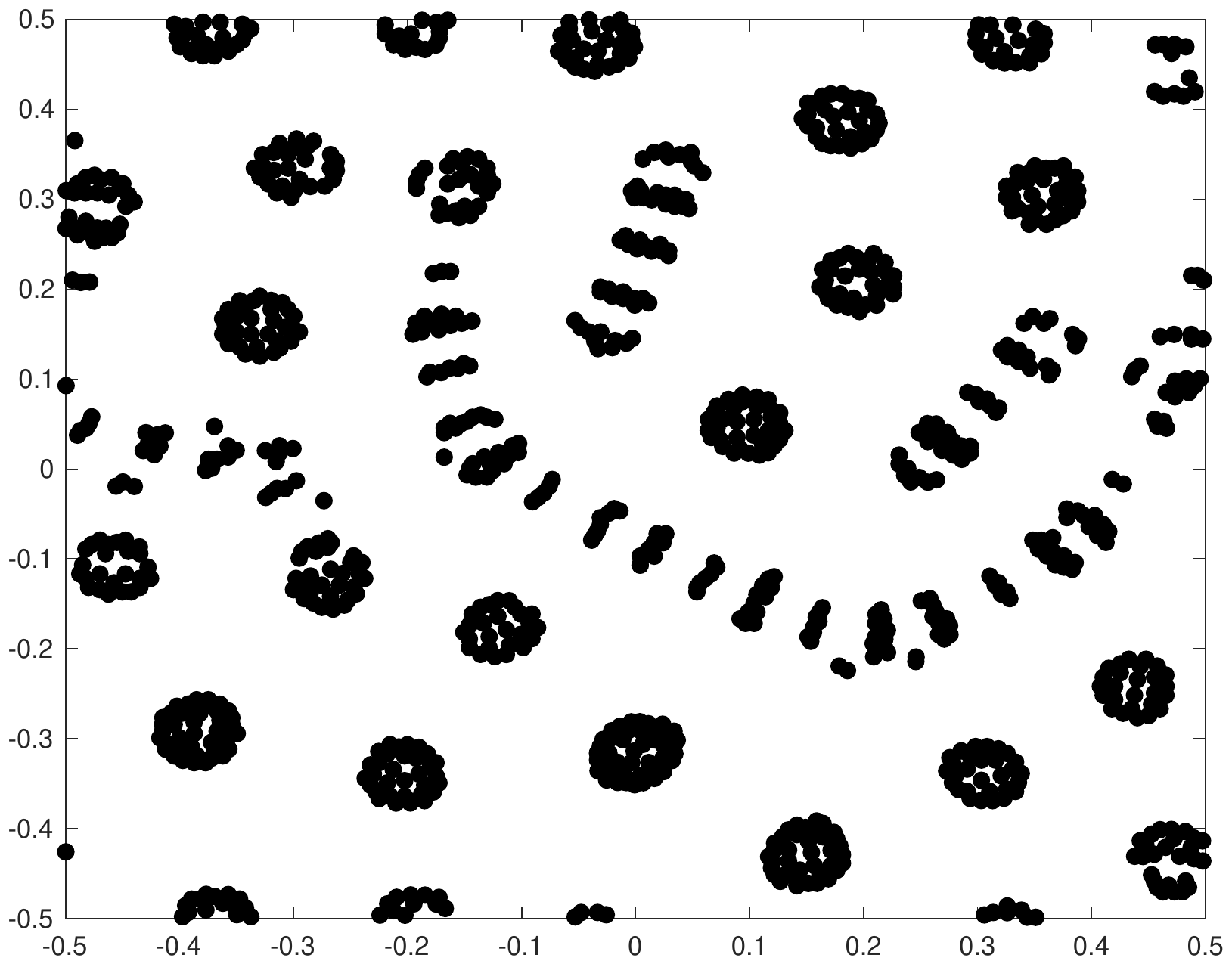}
		\includegraphics[width=.48\linewidth]{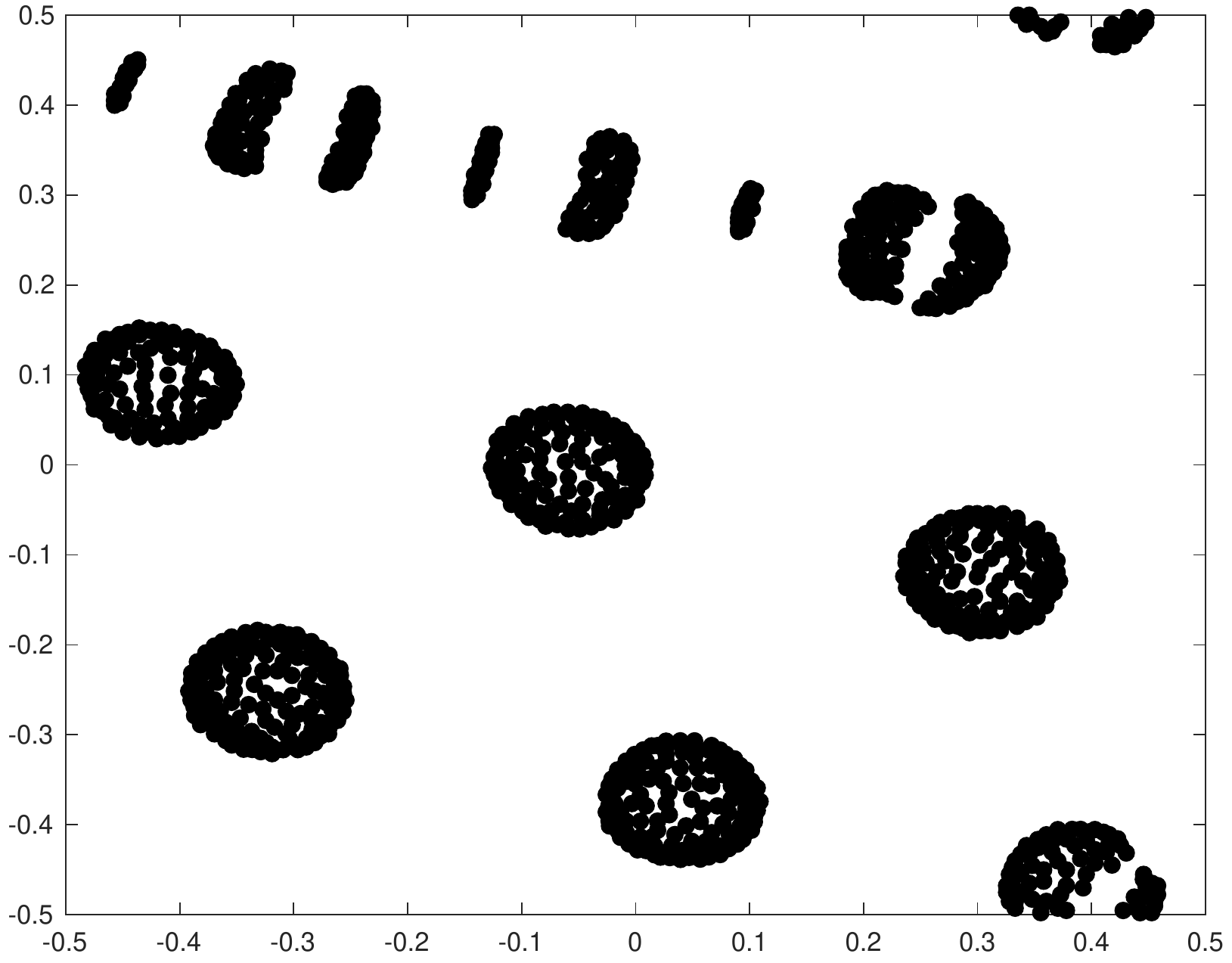}
		\caption{Long time simulation of the particle gradient flow $\dot{\bx}_i = -\frac{1}{N}\sum_{j\ne i} \nabla W(\bx_i-\bx_j),\,N=1000$. It (at least locally) minimizes the discrete version of $\cE_W$, namely, $\frac{1}{2N^2}\sum_{i\ne j} W(\bx_i-\bx_j)$. We take $d=2,\,s=-1$. Top left: $W=W_s$. For the other three pictures, $W$ is given by \eqref{Wc0} with $c_0=50$ and $\psi(\bx) = e^{-1/(1-|\bx|^2)}\chi_{|\bx|\le 1}$.  Top right: $\epsilon=0.05$; Bottom left: $\epsilon=0.1$; Bottom right: $\epsilon=0.2$. }
		\label{fig1}
	\end{center}
\end{figure}

Now we can see that Theorem \ref{thm4} gives a control on how wild the minimizers of $\cE_{\tilde{W}_\epsilon}$ could be. Although complicated structure can form at a fine level, any minimizer of $\cE_{\tilde{W}_\epsilon}$ have to remain close to the uniform distribution in the sense of the $d_\infty$ distance.

\subsection{Method}
In this section, we give a sketch of the proof for Theorem \ref{thm1}. 
\subsubsection{Wasserstein distance}
Our starting point is to give a convenient description for $d_{\infty}(\rho_1, \rho_2)$ in a general space so that we can use $d_{\infty}(\rho, 1)$ in $d>1$ in place for the discrepancy $\cD[\rho]$ in $d=1$ \eqref{eqn:D-H-def} . This description is inspired by the property of discrepancy in $d=1$.


For a given measurable set $S\subseteq \mathbb{T}^d$ and $r>0$, we denote the expansion of $S$ by $r$ as
\begin{equation}\label{Sr}
	S_r = \{\bx\in\mathbb{T}^d: \dist(\bx,S)<r\} = \bigcup_{\bx\in S} B(\bx;r).
\end{equation}
\begin{figure}
	\begin{center}
		\includegraphics[width=.6\linewidth]{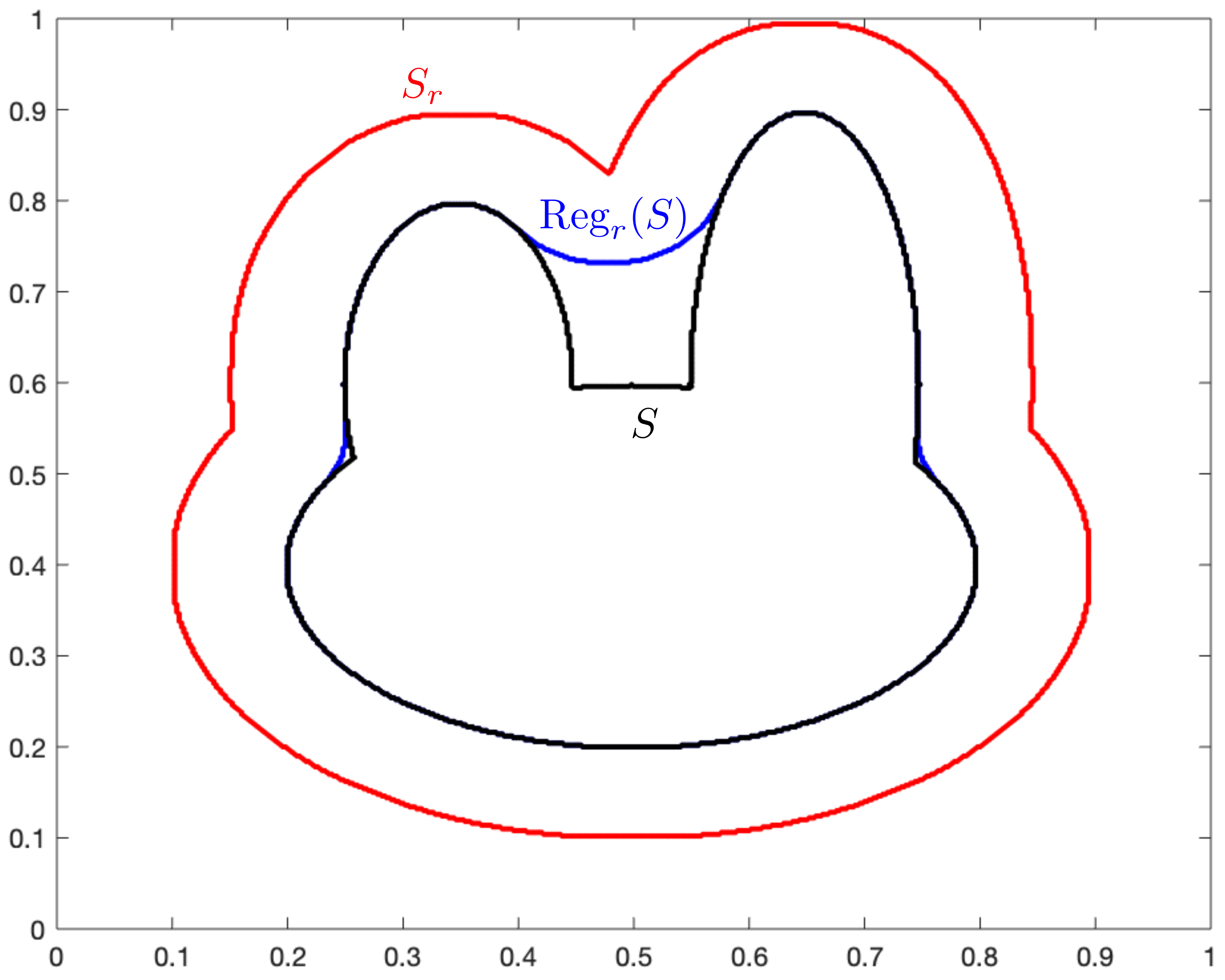}
		\caption{Illustration of $S_r$, the expansion of a set $S$, and the regularization $\Reg_r(S)$. $S$ is the region bounded by the innermost black curve. $\Reg_r(S)$, which contains $S$, is bounded by the blue curve. $S_r$ is the region bounded by the outermost red curve. }
		\label{fig2}
	\end{center}
\end{figure}
Then we  prove in Section \ref{sec_dinfty} the following theorem.
\begin{theorem}\label{lem_dinfty}
	Let $\rho_1,\rho_2\in \cM$. Then
	\begin{equation}\label{lem_dinfty_1}
		d_\infty(\rho_1,\rho_2) = \sup\Big\{r:\exists S\subseteq \mathbb{T}^d \textnormal{ s.t. }\int_S \rho_1\rd{\bx} > \int_{S_r}\rho_2\rd{\bx}\Big\}.
	\end{equation}
\end{theorem}

Recall that for the discrepancy $\cD[\rho]$ in $d =1$, if $I$ is an interval witnessing $\cD[\rho]$, then by \cite{Wass} one has $2d_\infty(\rho,1) = \cD[\rho] = \int_I(\rho-1)\rd{x} = \int_I \rho\rd{x} - \int_{I_{d_\infty(\rho,1)}}1\rd{x} + 2d_\infty(\rho,1)$, therefore $\int_I \rho\rd{x} = \int_{I_{d_\infty(\rho,1)}}1\rd{x}$. This shows that $I$, serving as $S$ in \eqref{lem_dinfty_1}, is almost the $S$ achieving the supremum for $d_\infty(\rho,1)$.

In order to to prove this theorem, we first prove the case of discrete measures using graph theory, namely, Hall's Theorem on perfect matchings of bipartite graphs. Then we generalize it to any probability measure by an approximation argument. 

\subsubsection{Fourier Analysis}
With Theorem \ref{lem_dinfty}, now we can apply the method in \cite{Sound} using Fourier analysis. When $d=1$, 
%
this method utilizes the interval $I$ witnessing $\cD[\rho]$, and consider a test function $g = \psi_r * \chi_{I_{r}}$ which is a mollified version of $\chi_{I_r}$. The test function $g$ captures the discrepancy between $\rho$ and the uniform distribution. Then the inequality in \cite{Sound} can be obtained by estimating $\int g\cdot (\rho-1)\rd{x}$ from below by $\cD[\rho]$, and from above by $\|W_{\log}*\rho\|_{L^1}$.

In Section \ref{sec_1pd} we generalize the method of Fourier analysis in \cite{Sound} to higher dimensions. To illustrate the idea, we denote $3r=d_\infty(\rho,1)$, and a mollifier $\psi_{r}$ supported in $B(0;r)$. In the spirit of the original method,  one takes a test function $g=\chi_{S_{r}}*\psi_{r}$, with $S$ maximizing in \eqref{lem_dinfty_1}. $g=1$ on $S$ while $\supp g\subseteq S_{2r}$. Such a test function enables us to detect the distance $d_\infty(\rho,1)$ by giving a positive lower bound $|S_{3r}\backslash S_{2r}|$ for $\int_{\mathbb{T}^d}g\cdot(\rho-1)\rd{\bx}$. Then, using standard Fourier analysis, we give an upper bound of the same quantity in terms of $\|W*\rho\|_{L^p}$ and $|S_{r}|$. Here, $|S_r|$ can be bounded from above by $|S_{2r}\backslash S_r|$ via an isoperimetric inequality, see Lemma \ref{lem_iso} where we give an isoperimetric inequality on $\bT^d$. We are now able to get an inequality involving $|S_{3r}\backslash S_{2r}|$, $|S_{2r}\backslash S_{r}|$ and $\| W*\rho \|_{L^p}$.

\subsubsection{Regularization}
Unlike in $d=1$ where $S$ can be taken to be closed intervals, the geometry for measurable subsets of $\R^d$ or $\bT^d$ will become much more complicated. 

We notice that there is a mismatch between the layers $|S_{3r}\backslash S_{2r}|$ and $|S_{2r}\backslash S_r|$, which could bring huge error if $S$ is complicated. Generally specking, $|S_{2r}\backslash S_r|$ could be much larger than $|S_{3r}\backslash S_{2r}|$. We overcome this difficulty by refining the choice of the test function. In fact, we apply a \emph{regularization} procedure to the set $S_r$ which will be used to define the test function $g$ (c.f. \eqref{g}) and play a subtle role in the proof of Theorem \ref{thm1} in Section \ref{sec_1pd}. We denote $S_r^c$ to be the complement of $S_r$, where $S_r$ is defined in \eqref{Sr}. 
\begin{definition}\label{def_reg}
	For $S\subseteq \mathbb{T}^d$ and $r>0$, define the \emph{$r$-regularization of $S$} as $\Reg_r(S) = (S_r^c)_r^c$.
\end{definition}
See Figure \ref{fig2} as an illustration. See Appendix \ref{sec_reg} for more properties of regularization which are not used in Section \ref{sec_1pd}. 

This operation removes all possible fine structures of $S_r$ at the scale $r$, and enables us to compare the expanded layers of the regularized $S_r$.

For $d=1$ case of Theorem \ref{thm1}, in Section \ref{sec_d1} we utilize the fact that $S$ can be taken as an interval, and conduct a slightly different way of bounding $\int_{\mathbb{T}}g\cdot(\rho-1)\rd{\bx}$ from above to treat a wider range of $p$.

Finally, for Theorem \ref{thm2}, in Section \ref{sec_opt} we construct an explicit class of measures $\rho$ in \eqref{rhocons}, whose $d_\infty(\rho,1)$ is clear, and the corresponding $\|W_s*\rho\|_{L^p}$ can be estimated by using the explicit structure given in Lemma \ref{lem_period}.
\subsection{Notations}
Throughout the paper, we denote $d$ to be the dimension, and $\bT^d = (\R/\mathbb{Z})^d$ to be the dimension $d$ torus. We denote $\cM(\bT^d)$ to be the set of probability measures on $\bT^d$. When there is no confusion, we will suppress $\bT^d$ and just write $\cM$. We will denote $1_{\bT^d}$ to be the uniform distribution. 

For $d = 1$, we denote $W_{\log}(x): = -\log|2\sin \pi x|$ to be the logarithmic potential. For $s<d$, we denote the Riesz potential by $W_s(x)$, which is defined by $\hat{W_s}(\bk) = |\bk|^{-d+s}$ for all $\bk\neq 0$ and $\hat{W_s}(0) = 0$.
For $\rho \in \cM(\bT^d)$, we define $V_W[\rho]:= W*\rho$ to be the potential generated by $\rho$ with the interaction potential $W$, and $\cE_W[\rho]:=\frac{1}{2}\int_{\bT^d} (W*\rho) \cdot \rho \rd{x}$ to be the potential energy of $\rho$. Given $\rho_1, \rho_2 \in \cM(\bT^d)$, we denote the Wasserstein-infinity distance between them by $d_{\infty}(\rho_1, \rho_2)$. For $\rho\in \cM(\bT)$, we denote $\cD[\rho]:= \sup_I \int_I (\rho-1)\rd{x}$ to be the discrepancy of $\rho$, where the supreme is taken over all closed intervals in $\bT$, and $\cH[\rho]:=\| (-W_{\log}*\rho)_+ \|_{L^{\infty}}$ to be the height of $\rho$.  

We denote the homogeneous Sobolev norm by $\|\cdot \|_{\dot{W}^{k, p}}$ and $\|\cdot \|_{\dot{H}^k}$ when $p=2$. For us, the Fourier transform (respectively Fourier coefficients) of $u$ is defined by
\begin{equation}
 \cF[u](\xi) =\hat{u}(\xi) = \int_{\mathbb{R}^d} u(\bx)e^{-2\pi i \xi\cdot\bx}\rd{\bx},\,\xi\in\mathbb{R}^d, \quad \cF[u](\bk) =\hat{u}(\bk) = \int_{\mathbb{T}^d} u(\bx)e^{-2\pi i \bk\cdot\bx}\rd{\bx},\,\bk\in\mathbb{Z}^d
\end{equation}
respectively when $u$ is a function over $\R^d$ (respectively $\bT^d$).

Any subset $S$ of $\mathbb{T}^d$ or $\mathbb{R}^d$ appearing in this paper will be assumed to be measurable. We denote $S_r$ to be the expansion of a set $S$ by radius $r$, defined in \eqref{Sr}. Throughout the paper, $S_r^c$ always denotes the complement of $S_r$. $\Reg_r(S):= (S_r^c)_r^c$ is the $r$-regularization of $S$. We will also say a set $S\subseteq \mathbb{T}^d$ is \emph{$r$-regular} if $S=\Reg_r(S)$.
\section{Equivalent formulation of Wasserstein-infinity distance}\label{sec_dinfty}
In this section we prove Theorem \ref{lem_dinfty}. In \eqref{lem_dinfty_1}, one clearly has $\text{LHS}\ge\text{RHS}$. In fact, let $r>0$ satisfies the condition on the RHS, i.e., there exists $S\subseteq \mathbb{T}^d$ such that
\begin{equation}\label{S_Sr}
	\int_S\rho_1(\bx)\rd{\bx} > \int_{S_r}\rho_2(\by)\rd{\by}.
\end{equation}
Then for any transport plan $\mu(\bx,\by)$ from $\rho_1$ to $\rho_2$, we have
\begin{equation}
	I_1 := \int_S\rho_1(\bx)\rd{\bx} = \int_S \int_{\mathbb{T}^d}\mu(\bx,\by)\rd{\by}\rd{\bx} ,\quad I_2 := \int_{S_r}\rho_2(\by)\rd{\by}
	= \int_{\mathbb{T}^d}\int_{S_r}\mu(\bx,\by)\rd{\by}\rd{\bx} .
\end{equation}
This implies that $\supp(\mu(\bx,\by)\chi_S(\bx))\not\subseteq S\times S_r$, because otherwise 
$I_1 = \int_S \int_{S_r}\mu(\bx,\by)\rd{\by}\rd{\bx} \le I_2$ which contradicts \eqref{S_Sr}. Therefore we get $d_\infty(\rho_1,\rho_2)\ge r$ which proves the claim.

To deal with the other direction $\text{LHS}\le\text{RHS}$ for \eqref{lem_dinfty_1}, we need the following lemma which is a weighted version of Hall's Theorem.
\begin{lemma}\label{lem_hall}
	Let $a_1,\dots,a_n$, $b_1,\dots,b_m$ be positive real numbers with $\sum a_i = \sum b_i = 1$. Let $(V,\cE)$, $V=X\cup Y =\{x_1,\dots,x_n\}\cup \{y_1,\dots,y_m\}$ be a bipartite graph. Then the following are equivalent:
	\begin{itemize}
		\item There exists an $n\times m$ nonnegative matrix $\{c_{ij}\}$, such that
		\begin{equation}
			c_{ij}=0,\,\forall(x_i,y_j)\notin \cE,\quad \sum_i c_{ij} = b_j,\,j=1,\dots,m,\quad \sum_j c_{ij} = a_i,\,i=1,\dots,n.
		\end{equation}
		\item For all subsets $S\subseteq X$, $\sum_{x_i\in S} a_i \le \sum_{y_j\in \cN(S)} b_j$ where $\cN(S)$ denotes the neighborhood of $S$ in the graph $(V,\cE)$.
	\end{itemize}
\end{lemma}

\begin{proof}
	Clearly item 1 implies item 2, since $\sum_{x_i\in S} a_i = \sum_{x_i\in S}\sum_j c_{ij} = \sum_{x_i\in S}\sum_{y_j\in \cN(S)} c_{ij}  \le \sum_{y_j\in \cN(S)} b_j$. To prove the converse, assume the opposite of item 1. Let $\{c_{ij}\}$ be the maximizer of $\sum_{ij}c_{ij}$ in the set of nonnegative matrices satisfying
	\begin{equation}\label{cij}
		c_{ij}=0,\,\forall(x_i,y_j)\notin \cE,\quad \sum_i c_{ij} \le b_j,\,j=1,\dots,m,\quad \sum_j c_{ij} \le a_i,\,i=1,\dots,n.
	\end{equation}
	then $\sum_{ij}c_{ij}<1$, which implies that there exists $i_0$ such that $\sum_j c_{i_0j} < a_{i_0}$. Define $S_0=\{x_{i_0}\}$, and we iteratively define $S_k\subseteq X,T_k\subseteq Y$ as follows:
	\begin{itemize}
		\item $T_k = \cN(S_{k-1})$.
		\item $S_k =  S_{k-1}\cup \cN_c(T_k)$ where $\cN_c(T_k) = \{x_i: \text{ there exists } y_j\in T_k \text{ such that }c_{ij}>0\}$.
		\item If $(S_k,T_k)=(S_{k-1},T_{k-1})$ for some $k$, then the iteration stops.
	\end{itemize}
	It is clear that $\{S_k\}$ and $\{T_k\}$ are nondecreasing sequence of sets, and therefore the iteration stops at some finite $k$. 
	
	We claim that every $y_j\in T_k$ satisfies $\sum_i c_{ij} =  b_j$. Otherwise, let $k$ be the first time there exists $y_{j_k}\in T_k$ with $\sum_i c_{ij_k} < b_{j_k}$, and then $\sum_i c_{ij}=b_j>0$ for any $y_j\in T_l,\,l=1,\dots,k-1$. Then by the iteration procedure, we have a sequence of distinct elements
	\begin{equation}
		x_{i_0}, y_{j_1},x_{i_1}, y_{j_2},\dots,x_{i_{k-1}},y_{j_k},
	\end{equation}
	such that $(x_{i_l},y_{j_{l+1}})\in \cE,\,l=0,1,\dots,k-1$ and $c_{i_l j_l}>0,\,l=1,2,\dots,k$. Then, we define $\{\tilde{c}_{ij}\}$ being the same as $\{c_{ij}\}$ except the changes
	\begin{equation}
		\tilde{c}_{i_l j_{l+1}} = c_{i_l j_{l+1}}+\epsilon,\,l=0,1,\dots,k-1,\quad \tilde{c}_{i_l j_l} = c_{i_l j_l}-\epsilon,\,l=1,2,\dots,k,
	\end{equation}
	for $\epsilon>0$ small. Then $\sum_j \tilde{c}_{i_0 j} = \sum_j c_{i_0 j}+\epsilon < a_{i_0}$, $\sum_i \tilde{c}_{ij_k} = \sum_i c_{ij_k} + \epsilon < b_{j_k}$, and all the other $\sum_i \tilde{c}_{ij}$ and $\sum_j \tilde{c}_{ij}$ are the same for those with $c_{ij}$. Therefore $\{\tilde{c}_{ij}\}$ also satisfies \eqref{cij} with $\sum_{ij}\tilde{c}_{ij} = \sum_{ij}c_{ij}+\epsilon$, contradicting the maximality of $\sum_{ij}c_{ij}$.
	
	Denote the final state of the iteration as $(S,T)$, then $T= \cN(S)$ and $ \cN_c(T)\subseteq S$, and every $y_j\in T$ satisfies $\sum_i c_{ij} =  b_j$. Then
	\begin{equation}
		\sum_{x_i\in S} a_i > \sum_{x_i\in S} \sum_j c_{ij} = \sum_{x_i\in S} \sum_{y_j\in \cN(S)} c_{ij} =  \sum_{y_j\in \cN(S)} \sum_{x_i\in S}c_{ij} =  \sum_{y_j\in \cN(S)} \sum_i c_{ij} =  \sum_{y_j\in \cN(S)} b_j,
	\end{equation}
	where the first inequality uses \eqref{cij} and $x_{i_0}\in S$; the second equality uses the fact that $y_j\in \cN(S)=T$ and $\cN_c(T)\subseteq S$. Therefore this $S$ contradicts item 2 in the statement of the lemma.
	
\end{proof}

\begin{remark}
	When the weights $a_i,b_j$ are all rational numbers, Lemma \ref{lem_hall} is a direct consequence of the classical Hall's Theorem \cite{Hall}. However, it is necessary for us to treat the case of irrational weights, because there exist probability measures $\rho$ which cannot be approximated by empirical measures in \eqref{lem_dinfty1_1} with rational weights in the sense of the $d_\infty$ distance. It is clear that $\rho=\frac{1}{\sqrt{2}}\delta(x) + (1-\frac{1}{\sqrt{2}})\delta(x-\frac{1}{2})$ in 1D is such an example.
\end{remark}

This lemma allows us to prove Theorem \ref{lem_dinfty} in the case of weighted empirical measures.

\begin{lemma}\label{lem_dinfty1}
Theorem \ref{lem_dinfty} holds if 
	\begin{equation}\label{lem_dinfty1_1}
		\rho_1(\bx) = \sum_{i=1}^n a_i\delta(\bx-\bx_i),\quad \rho_2(\bx) = \sum_{i=1}^m b_j\delta(\bx-\by_j),
	\end{equation}
	for some $n,m\in\mathbb{N}$, $\bx_1,\dots,\bx_n,\by_1,\dots,\by_m\in\mathbb{T}^d$, $a_i,b_j>0$, $\sum_i a_i=\sum_j b_j = 1$. Furthermore, in this case the supremum on the RHS of \eqref{lem_dinfty_1} can be achieved.
\end{lemma}

See Figure \ref{fig3} as an illustration.

\begin{figure}
	\begin{center}
		\includegraphics[width=.45\linewidth]{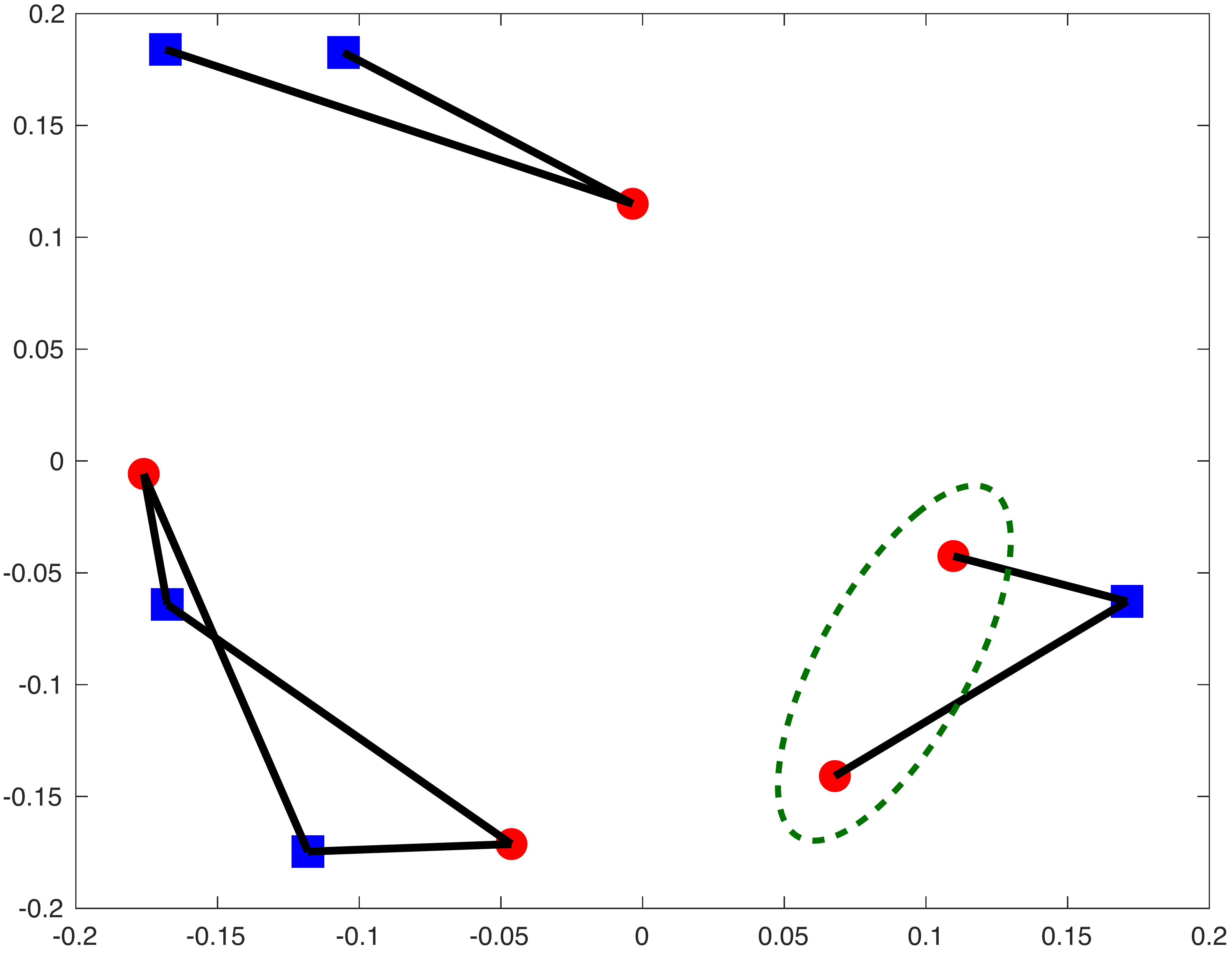}
		\includegraphics[width=.45\linewidth]{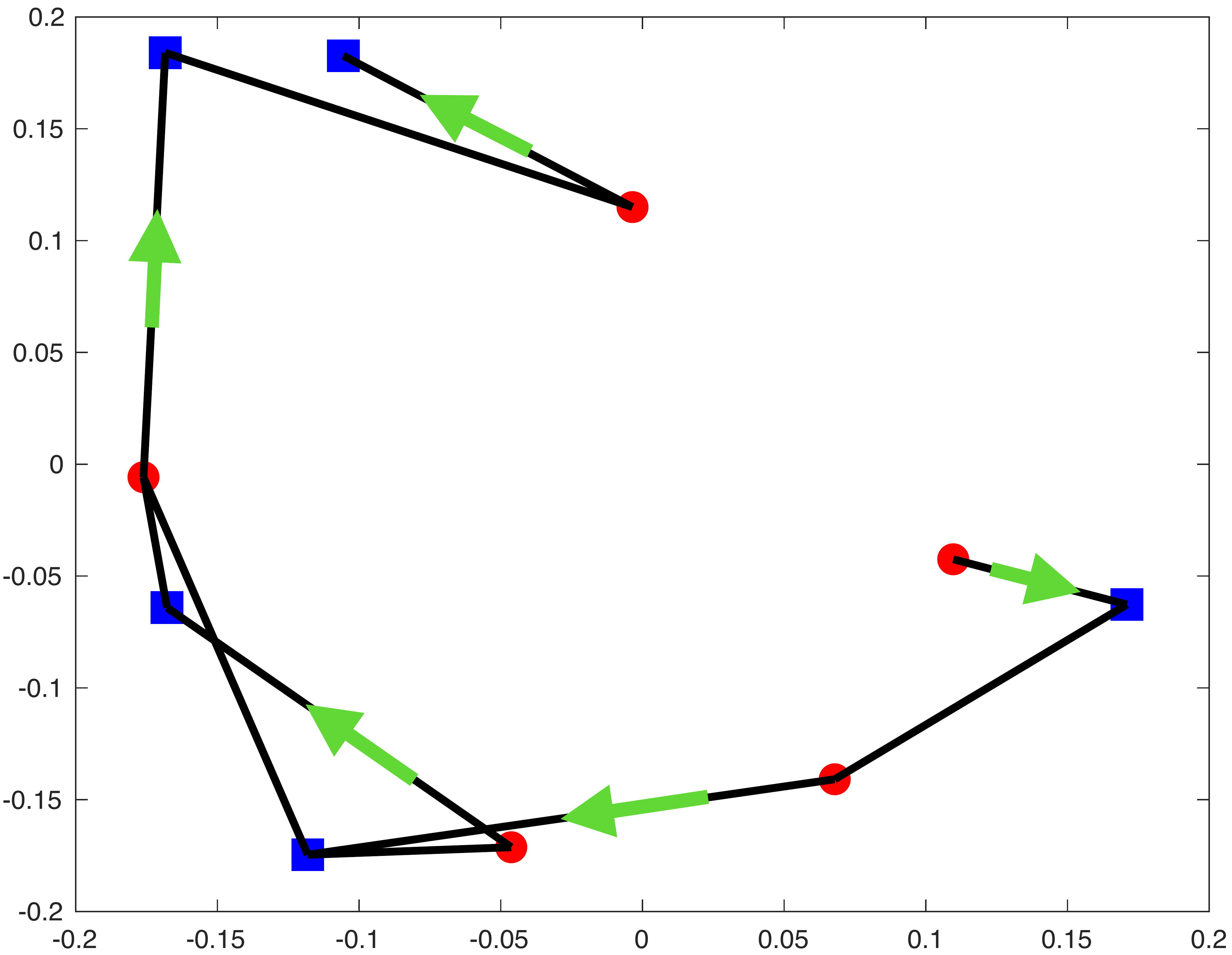}
		\caption{Proof of Lemma \ref{lem_dinfty1}. Here $\rho_1=\frac{1}{N}\sum_{j=1}^N \delta(\bx-\bx_j),\,N=5$, is shown by the red spots, and $\rho_2=\frac{1}{N}\sum_{j=1}^N \delta(\bx-\by_j)$ is shown by the blue squares. The two pictures are the graph $(V,\cE_r)$ (with the black segments represent the edges), for $r=0.18$ and $r=0.19$ respectively. $(V,\cE_{0.18})$ does not admit a perfect matching because the set of red spots in the dashed circle violates item 2 of Lemma \ref{lem_hall}. $(V,\cE_{0.19})$ admits a perfect matching, as indicated by the green arrows, which provide a transport plan from $\rho_1$ to $\rho_2$. Therefore one can conclude $0.18 \le d_\infty(\rho_1,\rho_2) < 0.19$. }
		\label{fig3}
	\end{center}
\end{figure}

\begin{proof}
	
	It suffices to prove the $\le$ direction in \eqref{lem_dinfty_1}. For $r>0$, define a bipartite graph $(V,\cE_r)$ by
	\begin{equation}
		V = \{\bx_1,\dots,\bx_n\}\cup  \{\by_1,\dots,\by_m\},\quad \cE_r = \{(\bx_i,\by_j): |\bx_i-\by_j|<r\}.
	\end{equation}
	Notice that the $\{c_{ij}\}$ in Lemma \ref{lem_hall}, if exists, would provide a transport plan $\mu(\bx,\by)=\sum_{ij} c_{ij}\delta(\bx-\bx_i)\delta(\by-\by_j)$ with $\max_{(\bx,\by)\in\supp\mu} |\bx-\by|<r$. Therefore, taking $r= d_\infty(\rho_1,\rho_2)$, there does not exist such $\{c_{ij}\}$ by the definition of $d_\infty$. By Lemma \ref{lem_hall}, this implies the existence of $S\subseteq \{\bx_1,\dots,\bx_N\}$ such that
	\begin{equation}
		\sum_{\bx_i\in S} a_i > \sum_{\by_j\in \cN(S)} b_j,
	\end{equation}
	that is,
	\begin{equation}
		\int_S \rho_1\rd{\bx} > \int_{S_r}\rho_2\rd{\bx},
	\end{equation}
	Therefore we get 
	\begin{equation}
		d_\infty(\rho_1,\rho_2) \le \sup\Big\{r:\exists S \text{ s.t. }\int_S \rho_1\rd{\bx} > \int_{S_r}\rho_2\rd{\bx}\Big\}.
	\end{equation}
	and the supremum on the RHS can be achieved, which finishes the proof.
	
\end{proof}

\begin{proof}[Proof of Theorem \ref{lem_dinfty}]
	
	It suffices to prove the $\le$ direction in \eqref{lem_dinfty_1}. Identify $\mathbb{T}^d$ as $[0,1)^d$, and denote $G_N=\{0,\frac{1}{N}\dots,\frac{N-1}{N}\}^d$ as the set of grid points for $N\in\mathbb{N}$. Define
	\begin{equation}
		\rho_{1,N}(\bx) = \sum_{\bj\in G_N} m_{1,\bj}\delta(\bx-\bj) \in \cM,\quad m_{1,\bj} := \int_{\by-\bj\in [0,\frac{1}{N})^d}\rho(\by)\rd{\by},
	\end{equation}
	as an approximation of $\rho_1$, and similarly define $\rho_{2,N}$. It is clear that
	\begin{equation}\label{rho1N}
		d_\infty(\rho_1,\rho_{1,N})\le \frac{\sqrt{d}}{N},\quad d_\infty(\rho_2,\rho_{2,N})\le \frac{\sqrt{d}}{N}.
	\end{equation}
	
	Then applying Lemma \ref{lem_dinfty1} gives
	\begin{equation}
		d_\infty(\rho_{1,N},\rho_{2,N}) = \max\Big\{r:\exists S \text{ s.t. }\int_S \rho_{1,N}\rd{\bx} > \int_{S_r}\rho_{2,N}\rd{\bx}\Big\}.
	\end{equation}
	Therefore there exists some set $S^{(N)}$, which is a subset of $G_N$ due to the proof of Lemma \ref{lem_dinfty1}, such that 
	\begin{equation}\label{rho1N_1}
		\int_{S^{(N)}} \rho_{1,N}\rd{\bx} > \int_{S_{r_N}^{(N)}}\rho_{2,N}\rd{\bx},\quad r_N=d_\infty(\rho_{1,N},\rho_{2,N}) \ge d_\infty(\rho_1,\rho_2)-\frac{2\sqrt{d}}{N}.
	\end{equation}

	Define 
	\begin{equation}
		\tilde{S}^{(N)} = \Big\{\bx: \bx-\bj\in [0,\frac{1}{N})^d \text{ for some }\bj\in S^{(N)}\Big\}.
	\end{equation}
	Then it is clear that  
	\begin{equation}\label{rho1N_2}
		\int_{S^{(N)}}\rho_{1,N}\rd{\bx} = \int_{\tilde{S}^{(N)}}\rho_1\rd{\bx},\quad \tilde{S}^{(N)}\subseteq (S^{(N)})_{\sqrt{d}/N},
	\end{equation}
	where the subscript $\sqrt{d}/N$ is interpreted as in \eqref{Sr}. Also, by \eqref{rho1N} and the $\ge$ direction of \eqref{lem_dinfty_1}, we have
	\begin{equation}
		\int_{T}\rho_{2}\rd{\bx} \le \int_{T_{\sqrt{d}/N+\epsilon}}\rho_{2,N}\rd{\bx},
	\end{equation}
	for any $T\subseteq \mathbb{T}^d$ and $\epsilon>0$. Applying this with $T=(\tilde{S}^{(N)})_{r-2\sqrt{d}/N}$ and $\epsilon=\sqrt{d}/N$, we get
	\begin{equation}
		\int_{(\tilde{S}^{(N)})_{r}}\rho_{2,N}\rd{\bx} \ge \int_{(\tilde{S}^{(N)})_{r-2\sqrt{d}/N}}\rho_2\rd{\bx},
	\end{equation}
	for any $r\ge2\sqrt{d}/N$. Therefore, combined with \eqref{rho1N_2} and \eqref{rho1N_1}, we get
	\begin{equation}
		\int_{\tilde{S}^{(N)}} \rho_1\rd{\bx} = \int_{S^{(N)}}\rho_{1,N}\rd{\bx} > \int_{(S^{(N)})_{r_N}}\rho_{2,N}\rd{\bx} \ge \int_{(\tilde{S}^{(N)})_{r_N-\sqrt{d}/N}}\rho_{2,N}\rd{\bx} \ge   \int_{(\tilde{S}^{(N)})_{r_N-3\sqrt{d}/N}}\rho_2\rd{\bx},
	\end{equation}
	where the second inequality uses the fact $(\tilde{S}^{(N)})_{r_N-\sqrt{d}/N}\subseteq (S^{(N)})_{r_N}$, coming from $\tilde{S}^{(N)}\subseteq (S^{(N)})_{\sqrt{d}/N}$ in \eqref{rho1N_2}. This implies that the RHS of \eqref{lem_dinfty_1} is at least $r_N-\frac{3\sqrt{d}}{N} \ge  d_\infty(\rho_1,\rho_2)-\frac{5\sqrt{d}}{N}$. Sending $N\rightarrow\infty$, we get the conclusion.
\end{proof}

\begin{remark}
	From the proof, it is clear that Theorem \ref{lem_dinfty} is also true if $\mathbb{T}^d$ is replaced by any compact Riemannian manifold, or more generally, any locally compact Riemannian manifold with $\rho_1$ and $\rho_2$ compactly supported.
\end{remark}

\section{Proof of Theorem \ref{thm1}, the case $1\le p \le d$}\label{sec_1pd}

In this section we prove Theorem \ref{thm1} in the case $1\le p \le d$. We first need an isoperimetric inequality. The classical isoperimetric inequality \cite{Oss} takes the form $|S_r\backslash S| \ge c r |S|^{(d-1)/d}$ for any bounded set $S\subseteq \mathbb{R}^d$ and $r>0$. However, for a set $S\subseteq \mathbb{T}^d$, it may happen that both $|S^c|$ and $|S_r\backslash S|$ are small but $|S|=O(1)$. Therefore we need an improvement which takes the following form.
\begin{lemma}\label{lem_iso}
	Let $S$ be a nonempty subset of $\mathbb{T}^d$, and $r>0$. Assume $S_r^c\ne\emptyset$. Then
	\begin{equation}
		|S_r\backslash S| \ge c r \min\{|S|,|S^c|\}^{\frac{d-1}{d}}.
	\end{equation}
\end{lemma}

We also need a lemma on the layers of expansions of a set $S$.
\begin{lemma}\label{lem_layer}
	For any $S\subseteq \mathbb{T}^d$,
	\begin{equation}
		|S_{2r}\backslash S_r| \le C|S_r\backslash S|,
	\end{equation}
	with $C$ only depending on $d$.
\end{lemma}
The proofs of both lemmas are in the Appendix.
%

Let $\psi:\mathbb{R}^d\rightarrow \mathbb{R}$ be a nonnegative smooth radial function supported inside $B(0;1)$ with $\int_{\mathbb{R}^d}\psi\rd{\bx}=1$, and denote $\psi_\epsilon(\bx) = \frac{1}{\epsilon^d}\psi(\frac{\bx}{\epsilon})$ for $0<\epsilon< 1/2$. $\psi$ is a radial function in the Schwartz class, and so is $\hat{\psi}$, with $\hat{\psi}(0)=1$. $\psi_\epsilon$ can also be viewed as a smooth function on $\mathbb{T}^d$ (identified with $[-1/2,1/2)^d$), still denoted as $\psi_\epsilon $, whose Fourier coefficients are given by the values of $\hat{\psi}_\epsilon = \hat{\psi}(\epsilon \cdot)$ at integer points.

\begin{proof}[Proof of Theorem \ref{thm1}, $1\le p \le d$ case]
	
	In this proof we will write $W$ for $W_s$. We may assume $d_\infty(\rho,1)>0$. Apply Theorem \ref{lem_dinfty} to get a set $S$ such that 
	\begin{equation}\label{S3r}
		\int_S \rho\rd{\bx} > \int_{S_{3r}}1\rd{\bx} = |S_{3r}|,
	\end{equation}
	for some $ r>d_\infty(\rho,1)/(4\sqrt{d})$. We may assume $r<1/2$ because $d_\infty(\rho,1) \le \sqrt{d}/2$, and thus $\psi_r$ is well-defined on $\mathbb{T}^d$ and supported inside $B(0;r)$. 
	
     We	define the set and test function
%
	\begin{equation}\label{g}
	T:=S_{3r}^c, \quad	g = \chi_{\tilde{S}}*\psi_{r},\quad \tilde{S} := \Reg_{2r}(S_{r}) = T_{2r}^c,
	\end{equation}
	which is supported inside $\tilde{S}_r$ and takes values in $[0,1]$. It is clear that $\dist(\tilde{S},T)\ge 2r$. Therefore $\dist(\tilde{S}_{r},T)\ge r$. Therefore
	\begin{equation}\label{STeq1}
		\tilde{S}_{r}\cap T_{r} = \supp g \cap T_{r}=\emptyset.
	\end{equation}
	Also notice that for any $\bx\in S$, we have $B(\bx;r) \subseteq S_r \subseteq \tilde{S}$, and thus $g|_S = 1$. Therefore, combining with \eqref{S3r}, we get
	\begin{equation}\begin{split}
			\int_{\mathbb{T}^d} g(\bx)(\rho(\bx)-1)\rd{\bx} \ge & \int_S \rho(\bx)\rd{\bx} - \int_{\supp g} 1\rd{\bx} > |T^c| - |\supp g| = |T^c\backslash \supp g| \\
			\ge &  |T^c\backslash T_r^c| =  |T_{r}\backslash T|.
	\end{split}\end{equation}
	Lemma \ref{lem_layer} applied to $T$ gives
	\begin{equation}
		|T_{2r}\backslash T_{r}| \le C |T_{r}\backslash T|.
	\end{equation}
	By \eqref{STeq1}, we have
	\begin{equation}
		|T_{2r}\backslash T_{r}| = |\tilde{S}^c \backslash T_{r}| \ge |\tilde{S}^c \backslash \tilde{S}_{r}^c| = |\tilde{S}_{r}\backslash \tilde{S}|.
	\end{equation}
	Therefore we get the lower bound
	\begin{equation}\label{eqg0}
		\int_{\mathbb{T}^d} g(\bx)(\rho(\bx)-1)\rd{\bx} \ge c|\tilde{S}_{r}\backslash \tilde{S}|.
	\end{equation}
	
	Then we use Fourier expressions to give an upper bound
	\begin{equation}\label{eqg2}\begin{split}
			\int_{\mathbb{T}^d} g(\bx)(\rho(\bx)-1)\rd{\bx} = & \sum_{\bk\ne 0} \hat{g}(\bk)\bar{\hat{\rho}}(\bk) = \sum_{\bk\ne 0} \hat{\psi}(r\bk)\hat{\chi}_{\tilde{S}}(\bk)\bar{\hat{\rho}}(\bk) = \sum_{\bk\ne 0} \frac{\hat{\psi}(r\bk)}{\hat{W}(\bk)}\hat{\chi}_{\tilde{S}}(\bk)\bar{\hat{W}}(\bk)\bar{\hat{\rho}}(\bk) \\
			= & \int_{\mathbb{T}^d}\cF^{-1}\Big(\frac{\hat{\psi}(r\cdot)}{\hat{W}(\cdot)}\hat{\chi}_{\tilde{S}}(\cdot)\Big)(\bx) (W*\rho)(\bx)\rd{\bx} \\
			\le & \Big\|\cF^{-1}\Big(\frac{\hat{\psi}(r\cdot)}{\hat{W}(\cdot)}\hat{\chi}_{\tilde{S}}(\cdot)\Big)\Big\|_{L^q}\|W*\rho\|_{L^p},
	\end{split}\end{equation}
	with $1/p+1/q=1$. Here we used the fact that $\hat{W}(\bk)$ is real and nonzero for $\bk\ne 0$, and the $\bk=0$ coefficient of the quantity inside $\cF^{-1}$ is viewed as 0. Notice that
	\begin{equation}
		\cF^{-1}\Big(\frac{\hat{\psi}(r\cdot)}{\hat{W}(\cdot)}\hat{\chi}_{\tilde{S}}(\cdot)\Big) = \chi_{\tilde{S}} * u_{r} = \chi_{\tilde{S}^c} * u_{r},\quad u_r:= \cF^{-1}\Big(\frac{\hat{\psi}(r\cdot)}{\hat{W}(\cdot)}\Big)= \cF^{-1}\Big(\hat{\psi}(r\cdot)|\cdot|^{d-s}\Big),
	\end{equation}
	using the fact that $u_r$ is mean-zero on $\mathbb{T}^d$. Therefore, by Young's inequality,
	\begin{equation}\label{eqg3}
		\Big\|\cF^{-1}\Big(\frac{\hat{\psi}(r\cdot)}{\hat{W}(\cdot)}\hat{\chi}_{\tilde{S}}(\cdot)\Big)\Big\|_{L^q} \le \min\{\|\chi_{\tilde{S}}\|_{L^q},\|\chi_{\tilde{S}^c}\|_{L^q}\}\|u_{r}\|_{L^1} = \min\{|\tilde{S}|,|\tilde{S}^c|\}^{1/q}\|u_{r}\|_{L^1}.
	\end{equation}
	Lemma \ref{lem_u} stated below implies that $\|u_{r}\|_{L^1} \le Cr^{-d+s}$. Combined with \eqref{eqg0} and \eqref{eqg2}, we get
	\begin{equation}
		|\tilde{S}_{r}\backslash \tilde{S}| \le  C\|W*\rho\|_{L^p}\min\{|\tilde{S}|,|\tilde{S}^c|\}^{1/q}r^{-d+s},
	\end{equation}
	i.e.,
	\begin{equation}\label{Sar_p}
		\|W*\rho\|_{L^p} \ge c r^{d-s}\frac{|\tilde{S}_{r}\backslash \tilde{S}|}{\min\{|\tilde{S}|,|\tilde{S}^c|\}^{1/q}} \ge c r^{d-s+1} \min\{|\tilde{S}|,|\tilde{S}^c|\}^{\frac{d-1}{d}-\frac{1}{q}},
	\end{equation}
	by applying Lemma \ref{lem_iso} to $(\tilde{S},r)$. Notice that the last power
	\begin{equation}
		\frac{d-1}{d}-\frac{1}{q} = 1-\frac{1}{d}-\frac{1}{q}=\frac{1}{p}-\frac{1}{d} \ge 0,
	\end{equation}
	for $1\le p\le d$. Therefore, from the fact that $\min\{|\tilde{S}|,|\tilde{S}^c|\} \ge cr^d$ (since both contain at least a ball of radius $r$), we see that $\min\{|\tilde{S}|,|\tilde{S}|^c\}^{\frac{d-1}{d}-\frac{1}{q}}\ge (cr^d)^{\frac{d-1}{d}-\frac{1}{q}} = cr^{\frac{d}{p}-1}$. Therefore, \eqref{Sar_p} gives
	\begin{equation}
		\|W*\rho\|_{L^p} \ge c r^{d-s+\frac{d}{p}},
	\end{equation}
	which is the conclusion.
\end{proof}
\begin{remark}
	In the case $p=1$, one could simplify the proof by taking $g=\chi_{S_r}*\psi_r$, without using regularization, Lemma \ref{lem_layer} or Lemma \ref{lem_iso}. The reason is that the quantity $\min\{|\tilde{S}|,|\tilde{S}^c|\}$ does not appear in \eqref{Sar_p} since $q=\infty$, and one can directly bound the numerator (now $|S_{2r}\backslash S_r|$ after replacing $\tilde{S}$ by $S_r$) from below by $cr^d$ and finish the proof.
\end{remark}
\begin{remark}
	The use of regularization is essential here. It is worth noticing that $S$ is $3r$-regular does not imply $S_r$ being $2r$-regular. To see this, one can consider the example where $S$ is a set of two isolated points with distance $2\sqrt{5}r$.
\end{remark}
\begin{remark}
For $p>d\ge 2$, we have at least a trivial bound $d_\infty(\rho,1) \le C \|W_s*\rho\|_{L^d}^{1/(d+1-s)} \le C \|W_s*\rho\|_{L^p}^{1/(d+1-s)}$ since $\bT^d$ has finite measure. However we do not know whether this estimates is sharp in terms of the scaling.
\end{remark}

\begin{lemma}\label{lem_u}
	Let $1\le q \le \infty$, $0<\epsilon< 1/2$, $\beta \ge 0$, and 
	\begin{equation}
		u_\epsilon:= \cF^{-1}\Big(\hat{\psi}(\epsilon \bk)|\bk|^\beta\Big),
	\end{equation}
	be a function on $\mathbb{T}^d$. There holds
	\begin{equation}
		\|u_\epsilon\|_{L^q} \le C\epsilon^{-\beta-d/p},
	\end{equation}
	with $C$ independent of $\epsilon$ and $1/p+1/q=1$.
\end{lemma}

The proof of the $1\le p \le d$ case of Theorem \ref{thm1} only uses the case $q=1$ of Lemma \ref{lem_u}. The general case of Lemma \ref{lem_u} will be used in the next section.

\begin{proof}
	Define
	\begin{equation}
		u:=  \cF^{-1}\Big(\hat{\psi}(\xi)|\xi|^\beta\Big),
	\end{equation}
	as a function on $\mathbb{R}^d$. $u$ is well-defined and is in $L^\infty(\mathbb{R}^d)$ since $\hat{\psi}$ is in the Schwartz class $\mathcal{S}(\mathbb{R}^d)$, and $\beta\ge 0$. We first claim that $u\in L^1(\mathbb{R}^d)$. In fact, the case $\beta=0$ is trivial. To deal with the case $\beta>0$, we aim to derive an estimate
	\begin{equation}\label{udecay}
		|u(\bx)|\le C|\bx|^{-(d+\beta)} + C|\bx|^{-(d+1)},\quad \forall |\bx|\ge 1,
	\end{equation}
	which implies the claim. To get \eqref{udecay}, we take $\bx$ with $|\bx|\ge 1$, and write
	\begin{equation}
		u(\bx) = \int_{\mathbb{R}^d} \hat{\psi }(\xi)|\xi|^{\beta} e^{2\pi i \bx\cdot \xi}\rd{\xi} = \int_{|\xi|\le |\bx|^{-1}} + \int_{|\xi|> |\bx|^{-1}}.
	\end{equation}
	The first integral is estimated by
	\begin{equation}
		\left|\int_{|\xi|\le |\bx|^{-1}} \hat{\psi }(\xi)|\xi|^{\beta} e^{2\pi i \bx\cdot \xi}\rd{\xi}\right| \le C\int_{|\xi|\le |\bx|^{-1}} |\xi|^{\beta} \rd{\xi} = C|\bx|^{-(d+\beta)}.
	\end{equation}
	To estimate the second integral at $\bx=(x_1,\dots,x_d)$, we assume without loss of generality that $|x_1|\ge |\bx|/\sqrt{d}$, and then
	\begin{equation}\label{psibeta}\begin{split}
			\int_{|\xi|> |\bx|^{-1}} & \hat{\psi }(\xi)|\xi|^{\beta} e^{2\pi i \bx\cdot \xi}\rd{\xi} = \frac{1}{2\pi i x_1}\int_{|\xi|> |\bx|^{-1}} \hat{\psi }(\xi)|\xi|^{\beta} \partial_{\xi_1}e^{2\pi i \bx\cdot \xi}\rd{\xi} \\
			= & \frac{1}{2\pi i x_1}\int_{|\xi|= |\bx|^{-1}} \hat{\psi }(\xi)|\xi|^{\beta} e^{2\pi i \bx\cdot \xi}\bn(\xi)\cdot \vec{e}_1\rd{S(\xi)} -  \frac{1}{2\pi i x_1}\int_{|\xi|> |\bx|^{-1}} \partial_{\xi_1}(\hat{\psi }(\xi)|\xi|^{\beta} )e^{2\pi i \bx\cdot \xi}\rd{\xi} \\
			= & \cdots = \sum_{j=0}^d \frac{(-1)^j}{(2\pi i x_1)^{j+1}}\int_{|\xi|= |\bx|^{-1}} \partial_{\xi_1}^{j}(\hat{\psi }(\xi)|\xi|^{\beta}) e^{2\pi i \bx\cdot \xi}\bn(\xi)\cdot \vec{e}_1\rd{S(\xi)} \\
			& + \frac{(-1)^{d+1}}{(2\pi i x_1)^{d+1}}\int_{|\xi|> |\bx|^{-1}} \partial_{\xi_1}^{d+1}(\hat{\psi }(\xi)|\xi|^{\beta} )e^{2\pi i \bx\cdot \xi}\rd{\xi}.
	\end{split}\end{equation}
	where we use integration by parts $d+1$ times, and notice that there is no contribution from the boundary terms at infinity due to the fast decay of $ \hat{\psi}(\xi)|\xi|^\beta$ and its derivatives. 
	
	To estimate the RHS integrals in \eqref{psibeta}, we first notice that for any $0<|\xi|\le 1$,
	\begin{equation}\label{dpsixi}
		\big|\partial_{\xi_1}^j(\hat{\psi }(\xi)|\xi|^{\beta})\big| \le \sum_{k=0}^j {j\choose k} \big|\partial_{\xi_1}^{j-k}\hat{\psi }(\xi)\big|\cdot \big|\partial_{\xi_1}^k|\xi|^{\beta}\big| \le C\sum_{k=0}^j |\xi|^{\beta-k} \le C|\xi|^{\beta-j}.
	\end{equation}
	Therefore, for any $0\le j \le d$,
	\begin{equation}\begin{split}
			& \left|\frac{(-1)^j}{(2\pi i x_1)^{j+1}}\int_{|\xi|= |\bx|^{-1}} \partial_{\xi_1}^j(\hat{\psi }(\xi)|\xi|^{\beta}) e^{2\pi i \bx\cdot \xi}\bn(\xi)\cdot \vec{e}_1\rd{S(\xi)}\right| \\
			\le & C|\bx|^{-(j+1)}\cdot |\bx|^{-(\beta-j)}\cdot |\bx|^{-(d-1)} = C|\bx|^{-(d+\beta)}.
	\end{split}\end{equation}
	where we used \eqref{dpsixi} with $|\xi|=|\bx|^{-1}\le 1$.  Also, due to the fast decay of $ \partial_{\xi_1}^{d+1}(\hat{\psi }(\xi)|\xi|^{\beta} )$ at infinity, we may take a large $m$ and apply \eqref{dpsixi} again to get
	\begin{equation}\begin{split}
			& \left|\frac{(-1)^{d+1}}{(2\pi i x_1)^{d+1}}\int_{|\xi|> |\bx|^{-1}} \partial_{\xi_1}^{d+1}(\hat{\psi }(\xi)|\xi|^{\beta} )e^{2\pi i \bx\cdot \xi}\rd{\xi}\right| \\
			\le & C|\bx|^{-(d+1)}\int_{ |\bx|^{-1} < |\xi| \le 1} |\xi|^{\beta-d-1}\rd{\xi} + C|\bx|^{-(d+1)}\int_{ |\xi| > 1} |\xi|^{-m}\rd{\xi} \\
			\le & C|\bx|^{-(d+1)}(|\bx|^{-(\beta-1)} + 1) \le C|\bx|^{-(d+\beta)} + C|\bx|^{-(d+1)}.
	\end{split}\end{equation}
	Therefore we conclude \eqref{udecay}.
	
	Next we claim that
	\begin{equation}\label{claim_uT}
		u_\epsilon(\bx) = \epsilon^{-d-\beta} \sum_{\bj\in\mathbb{Z}^d}u(\frac{\bx-\bj}{\epsilon}),
	\end{equation}
	which would finish the proof since it implies
	\begin{equation}\begin{split}
			\|u_\epsilon\|_{L^1(\mathbb{T}^d)} = & \epsilon^{-d-\beta} \left\|\sum_{\bj\in\mathbb{Z}^d}u(\frac{\cdot-\bj}{\epsilon})\right\|_{L^1([-1/2,1/2)^d)} \le\epsilon^{-d-\beta} \sum_{\bj\in\mathbb{Z}^d}\left\|u(\frac{\cdot-\bj}{\epsilon})\right\|_{L^1([-1/2,1/2)^d)}  \\
			= & \epsilon^{-d-\beta} \|u(\cdot/\epsilon)\|_{L^1(\mathbb{R}^d)} = \epsilon^{-\beta} \|u\|_{L^1(\mathbb{R}^d)}.
	\end{split}\end{equation}
	Using \eqref{claim_uT}, \eqref{udecay} and $u\in L^\infty$, it is also clear that $\|u_\epsilon\|_{L^\infty} \le C\epsilon^{-d-\beta}$ for $0<\epsilon\le 1/2$. Therefore, by interpolation,
	\begin{equation}
		\|u_\epsilon\|_{L^q} \le \|u_\epsilon\|_{L^1}^{1/q} \|u_\epsilon\|_{L^\infty}^{1/p} \le C\epsilon^{-\beta - d/p}.
	\end{equation}

	To see \eqref{claim_uT}, we first recall the Poisson summation formula: let $\Phi(\bx) = \sum_{\bj\in\mathbb{Z}^d}\delta(\bx-\bj)$, then $\hat{\Phi}(\xi) = \sum_{\bj\in\mathbb{Z}^d}\delta(\xi-\bj)$. Denote $\tilde{u}_\epsilon$ as the periodic extension of $u_\epsilon$ to $\mathbb{R}^d$, then
	\begin{equation}
		\hat{\tilde{u}}_\epsilon(\xi) = \sum_{\bj\in\mathbb{Z}^d} \hat{u}_\epsilon(\bj)\delta(\xi-\bj)= \sum_{\bj\in\mathbb{Z}^d} \hat{\psi}(\epsilon \bj)|\bj|^{\beta}\delta(\xi-\bj) = \epsilon^{-\beta}\sum_{\bj\in\mathbb{Z}^d}\hat{u}(\epsilon \bj)\delta(\xi-\bj),
	\end{equation}
	i.e.,
	\begin{equation}
		\hat{\tilde{u}}_\epsilon(\xi) = \epsilon^{-\beta}\hat{u}(\epsilon \xi)\hat{\Phi}(\xi).
	\end{equation}
	Therefore
	\begin{equation}
		\tilde{u}_\epsilon(\bx) = \epsilon^{-\beta}\cF^{-1}(\hat{u}(\epsilon \cdot)) * \Phi= \epsilon^{-d-\beta}u(\cdot/\epsilon) * \Phi = \epsilon^{-d-\beta} \sum_{\bj\in\mathbb{Z}^d}u(\frac{\bx-\bj}{\epsilon}).
	\end{equation}
\end{proof}

%

\section{Proof of Theorem \ref{thm1}, the case $d=1$}\label{sec_d1}

In this section we prove Theorem \ref{thm1} in the case $d=1$. We recall from \eqref{D1} that $d_\infty(\rho,1)  = \frac{1}{2}\cD[\rho]$. Therefore, instead of taking a general $S$ which approximately achieves \eqref{lem_dinfty_1}, we may take a closed interval $I$ with $\int_I (\rho-1)\rd{\bx} = 2d_\infty(\rho,1)$. The function $\chi_{I_r}$ (for $r>0$) is much easier to deal with than a general $\chi_{S_r}$, and this allows us to gain improvement in the 1D case.

Then we need to enlarge the range of parameter in Lemma \ref{lem_u} for 1D.
\begin{lemma}\label{lem_u1}
	Let $d=1$, $1\le q \le \infty$, $0<\epsilon< 1/2$, $-1<\beta<0$, and
	\begin{equation}
		u_\epsilon:= \cF^{-1}\Big(\hat{\psi}(\epsilon k)|k|^\beta\Big),
	\end{equation}
	be a function on $\mathbb{T}$, with $\hat{u}_\epsilon(0)=0$. There holds
	\begin{equation}
		\|u_\epsilon\|_{L^q} \le C\left\{\begin{split}
			& \epsilon^{-\beta-1/p},\quad -\frac{1}{p}<\beta<0 \\
			& (1+|\log \epsilon|)^{1/q},\quad \beta=-\frac{1}{p}  \\
			& 1,\quad -1<\beta<-\frac{1}{p} 
		\end{split}\right.
	\end{equation}
	with $C$ independent of $\epsilon$ and $1/p+1/q=1$.
\end{lemma}

\begin{proof}
	By Lemma \ref{lem_period}, the periodized 1D Riesz kernel $W_s$ with $0<s<1$ is smooth on $\mathbb{T}\backslash \{0\}$, and differs from $c|x|^{-s}$ by a smooth function (identifying $\mathbb{T}=[-1/2,1/2)$) near 0. Therefore, we have
	\begin{equation}
		\frac{1}{2r}\int_{B(x;r)}|W_s(y)|\rd{y} \le C\max\{W_s(x),1\},\quad \frac{1}{2r}\int_{B(x;r)}|W_s(y)|\rd{y} \le Cr^{-s},
	\end{equation}
	for any $x\in\mathbb{T}$ and $0<r\le 1/2$, since the same property is clearly true for $y\mapsto |y|^{-s}$ near 0. Therefore
	\begin{equation}
		u_\epsilon= \psi_\epsilon*W_{s},\quad s=1+\beta \in (0,1),
	\end{equation}
	can be estimated by
	\begin{equation}
		|u_\epsilon(x)| \le C\left\{\begin{split}
			& \epsilon^{-s},\quad |x|<\epsilon \\
			& 1+|x|^{-s},\quad \epsilon\le |x| \le \frac{1}{2}
		\end{split}\right.
	\end{equation}
	Therefore
	\begin{equation}
		\|u_\epsilon\|_{L^q} \le C\epsilon^{-s+1/q} + C\left(\int_\epsilon^{1/2} (1+x^{-s q})\rd{x}\right)^{1/q}.
	\end{equation}
	Then we separate into cases:
	\begin{itemize}
		\item If $\beta>-\frac{1}{p} $, then $-s q = -q-\beta q < -q + q/p = -1$. Then
		\begin{equation}
			\int_\epsilon^{1/2} (1+x^{-s q})\rd{x} \le C\epsilon^{-s q +1},
		\end{equation}
		which gives $\|u_\epsilon\|_{L^q} \le C\epsilon^{-s+1/q} = C\epsilon^{-\beta-1/p}$.
		\item If $\beta=-\frac{1}{p} $, then $-s q = -1$. Then
		\begin{equation}
			\int_\epsilon^{1/2} (1+x^{-s q})\rd{x} \le C(1+|\log \epsilon|),
		\end{equation}
		which gives $\|u_\epsilon\|_{L^q} \le C(1+|\log \epsilon|)^{1/q}$.
		\item If $\beta<-\frac{1}{p} $, then $-s q > -1$. Then
		\begin{equation}
			\int_\epsilon^{1/2} (1+x^{-s q})\rd{x} \le C,
		\end{equation}
		which gives $\|u_\epsilon\|_{L^q} \le C$.
	\end{itemize}
	
\end{proof}

\begin{proof}[Proof of Theorem \ref{thm1}, $d=1$, and $1<p<\infty$ case]
	As discussed at the beginning of this section, we may take a closed interval $I$ with $\int_I \rho\rd{x} \ge |I_{3r}|$ with $r=d_\infty(\rho,1)/3$. Then, defining $g = \chi_{I_{r}}*\psi_{r}$, we get 
	\begin{equation}\label{1d_est1}
		\int_{\mathbb{T}}g(x)(\rho(x)-1)\rd{x} \ge 2r,
	\end{equation}
	and
	\begin{equation}\label{1d_est2}
		\int_{\mathbb{T}}g(x)(\rho(x)-1)\rd{x} \le  \Big\|\cF^{-1}\Big(\frac{\hat{\psi}(r\cdot)}{\hat{W}(\cdot)}\hat{\chi}_{I_{r}}(\cdot)\Big)\Big\|_{L^q}\|W*\rho\|_{L^p},
	\end{equation}
	similar to the proof of the $1\le p\le d$ case. Notice that
	\begin{equation}
		\chi_{I_{r}}' = \delta_{x_1}-\delta_{x_2},\quad\text{ where } [x_1,x_2]=I_{r},
	\end{equation}
	and
	\begin{equation}
		\hat{\chi}_{I_{r}}(k) = \frac{1}{2\pi i k}\widehat{\chi'_{I_{r}}}(k),\quad k\ne 0.
	\end{equation}
	Therefore
	\begin{equation}\begin{split}\label{Lq1d}
			\Big\|\cF^{-1}\Big(\frac{\hat{\psi}(r\cdot)}{\hat{W}(\cdot)}\hat{\chi}_{I_{r}}(\cdot)\Big)\Big\|_{L^q} = & \Big \|\cF^{-1}\Big(\frac{\hat{\psi}(r\cdot)}{2\pi i (\cdot)\hat{W}(\cdot)}\Big) * (\delta_{x_1}-\delta_{x_2})\Big\|_{L^q} \\
			\le & C\Big\|\cF^{-1}\Big(\hat{\psi}(r\cdot)|\cdot|^{-s}\sgn(\cdot)\Big) \Big\|_{L^q}\le C\Big\|\cF^{-1}\Big(\hat{\psi}(r\cdot)|\cdot|^{-s}\Big) \Big\|_{L^q},
	\end{split}\end{equation}
	where the last inequality uses the boundedness of the Hilbert transform on $L^q$. Then we apply Lemma \ref{lem_u} in the case $s\le 0$ and Lemma \ref{lem_u1} in the case $0<s<1$ and get
	\begin{equation}
		\Big\|\cF^{-1}\Big(\hat{\psi}(r\cdot)|\cdot|^{-s}\Big) \Big\|_{L^q} \le C\left\{\begin{split}
			& r^{s-1/p},\quad s<\frac{1}{p} \\
			& (1+|\log r|)^{1/q},\quad s=\frac{1}{p}  \\
			& 1,\quad \frac{1}{p} <s<1
		\end{split}\right.
	\end{equation}
	Combined with \eqref{1d_est1} and \eqref{1d_est2}, we get that $r$ is less than the above RHS times $\|W*\rho\|_{L^p}$, which is the conclusion.
	
\end{proof}

\section{Optimality of scaling}\label{sec_opt}
In this section we prove the following theorem, which indicates (iii) in Theorem \ref{thm1}.

\begin{theorem}\label{thm2}
	Let $0<s<d$ and $1\le p \le \infty$. There exists constant $c>0$, $\rho\in\cM$ with $d_\infty(\rho,1)$ arbitrarily small, such that
	\begin{equation}\label{thm2_1}
		d_\infty(\rho,1)\ge c \|W_s*\rho\|_{L^p}^{\min\{\gamma,1\}},\quad \gamma = \frac{1}{d+d/p-s}.
	\end{equation}
\end{theorem}

We first construct a microscopic profile by taking derivatives on the mollifier $\psi$ defined in Section \ref{sec_1pd}.
\begin{lemma}\label{lem_M}
	Let $M\in\mathbb{N}$, and define 
	\begin{equation}
		\Psi_M = (-\Delta)^M \psi, \quad \Psi_{M,\epsilon}(\bx) = \frac{1}{\epsilon^d} \Psi_M(\frac{\bx}{\epsilon}),
	\end{equation}
	as functions on $\mathbb{R}^d$. Then for any smooth function $f$ defined on $B(0;\epsilon)$,
	\begin{equation}\label{lem_M_1}
		\int_{B(0;\epsilon)}f(\bx)\Psi_{M,\epsilon}(\bx) \rd{\bx} \le C\epsilon^{2M}\|f\|_{W^{2M,\infty}(B(0;\epsilon))},
	\end{equation}
	where $C$ depends on $d,M,\psi$.
\end{lemma}

We will see in \eqref{eqM} that taking a multiple Laplacian on $\psi$ and use it in the construction of $\rho$ will make sure that $W*\rho$ concentrates near 0.

\begin{proof}
	We first claim that 
	\begin{equation}\label{claim_M}
		\int_{B(0;\epsilon)}f(\bx)\Psi_{M,\epsilon}(\bx) \rd{\bx}=0,
	\end{equation}
	if $f$ is a polynomial of degree no more than $2M-1$. To see this, we may assume $\epsilon=1$ and $f$ is a monomial $\bx^\alpha$ with the multi-index $|\alpha|\le 2M-1$ without loss of generality. Then $f \Psi_M$ is a  smooth function compactly supported in $B(0;1)$, whose integral is given by
	\begin{equation}
		\int_{B(0;1)}f(\bx)\Psi_{M}(\bx) \rd{\bx} = \int_{\mathbb{R}^d}\bx^\alpha\Psi_{M}(\bx) \rd{\bx} = \cF((\cdot)^\alpha\Psi_{M})(0) = \frac{1}{(-2\pi i)^{|\alpha|}} (\partial_\xi^\alpha \hat{\Psi}_M)(0).
	\end{equation}
	Notice that $\hat{\Psi}_M(\xi) = (2\pi)^{2M}|\xi|^{2M}\hat{\psi}(\xi)$ with $\hat{\psi}$ being a smooth function. Therefore $(\partial_\xi^\alpha \hat{\Psi}_M)(0)=0$ and the claim follows.
	
	To show Lemma \ref{lem_M_1}, we use Taylor expansion for $f$ at 0 to get
	\begin{equation}
		|f(\bx) - T_{2M-1}[f](\bx) | \le C|\bx|^{2M}\|f\|_{W^{2M,\infty}(B(0;\epsilon))}\le C\epsilon^{2M}\|f\|_{W^{2M,\infty}(B(0;\epsilon))},
	\end{equation}
	for any $|\bx|<\epsilon$, where $T_{2M-1}[f]$ is the Taylor polynomial of degree $2M-1$ of $f$. By \eqref{claim_M}, we get
	\begin{equation}\begin{split}
			\left|\int_{B(0;\epsilon)}f(\bx)\Psi_{M,\epsilon}(\bx) \rd{\bx}\right| = & \left|\int_{B(0;\epsilon)}(f(\bx)-T_{2M-1}[f](\bx))\Psi_{M,\epsilon}(\bx) \rd{\bx}\right| \\
			\le & C\epsilon^{2M}\|f\|_{W^{2M,\infty}(B(0;\epsilon))}\int_{B(0;\epsilon)}|\Psi_{M,\epsilon}(\bx) |\rd{\bx} \le  C\epsilon^{2M}\|f\|_{W^{2M,\infty}(B(0;\epsilon))}.
	\end{split}\end{equation}

\end{proof}

%
%
%

\begin{proof}[Proof of Theorem \ref{thm2}]
	
	It is easy to see the existence of $\rho\in \cM$ with $d_\infty(\rho,1)$ arbitrarily small, such that $d_\infty(\rho,1) \ge c\|W_s*\rho\|_{L^p}$. In fact, one can take $\rho=1- \epsilon |B(0;1/3)| +\epsilon \chi_{B(0;1/3)} $ for small $\epsilon>0$, and notice that $d_\infty(\rho,1)\ge c\epsilon$ (by Theorem \ref{lem_dinfty} with $S=B(0;1/3)$) and $\|W_s*\rho\|_{L^p}=C\epsilon$.
	
	Then we deal with the power $d+d/p-s$ in \eqref{thm2_1}. Let $M\in\mathbb{N}$ to be determined. For any small $\epsilon>0$, the function $\Psi_{M,\epsilon}$ in Lemma \ref{lem_M} can be viewed as a function on $\mathbb{T}^d$, and we construct
	\begin{equation}\label{rhocons}
		\rho=1 + c_0\epsilon^d\Psi_{M,\epsilon},\quad c_0 = \frac{1}{\|\Psi_M\|_{L^\infty}},
	\end{equation}
	which is clearly a probability measure on $\mathbb{T}^d$ with $\|\rho-1\|_{L^\infty}\le 1$ and $\supp(\rho-1) \subseteq B(0;\epsilon)$. It is clear that $d_\infty(\rho,1)\ge c\epsilon$ by Theorem \ref{lem_dinfty} with $S=B(0;\epsilon R)$ or its complement for some fixed $R\in (0,1)$ with $\int_{B(0;R)}\Psi_M\rd{\bx}\ne 0$. 
	
	Next we analyze $W_s*\rho$. If $|\bx|<2\epsilon$, then Lemma \ref{lem_period} and the fact $\|\rho\|_{L^\infty}\le 2$ show that
	\begin{equation}
		|(W_s*\rho)(\bx)| \le 2\int_{B(\bx;\epsilon)}W_s(\by)\rd{\by} \le C\epsilon^{d-s}.
	\end{equation}
	If $\bx\in [-1/2,1/2)^d$ with $|\bx|\ge 2\epsilon$ , then Lemmas \ref{lem_period} and \ref{lem_M} shows that
	\begin{equation}\label{eqM}
		|(W_s*\rho)(\bx)| \le C\epsilon^{2M+d} \|W_s\|_{W^{2M,\infty}(B(\bx;\epsilon))} \le C\epsilon^{d+2M}\cdot |\bx|^{-s-2M}.
	\end{equation}
	Therefore
	\begin{equation}
		\int_{|\bx|\le 2\epsilon}|(W_s*\rho)(\bx)|^p\rd{\bx} \le C\epsilon^{d+p(d-s)},
	\end{equation}
	and
	\begin{equation}
		\begin{split}
		&\int_{|\bx|> 2\epsilon}|(W_s*\rho)(\bx)|^p\rd{\bx} \le C\epsilon^{p(d+2M)}\int_{2\epsilon}^{\sqrt{d}} r^{p(-s-2M)+d-1}\rd{r}\\
		 \le &C\epsilon^{p(d+2M)}\cdot \epsilon^{p(-s-2M)+d} = C\epsilon^{d+p(d-s)},
		\end{split}
	\end{equation}
	by taking $M$ sufficiently large so that $p(-s-2M)+d-1<-1$. Therefore
	\begin{equation}
		\|W_s*\rho\|_{L^p} \le C\epsilon^{d+d/p-s}.
	\end{equation}
	This proves the conclusion with the power $d+d/p-s$, and finishes the proof.
	
\end{proof}

\section{Acknowledgement}
The first author was supported in part by NSF and ONR grants DMS1613911 and N00014-1812465. The first author was supported by the Advanced Grant Nonlocal-CPD (Nonlocal PDEs for Complex Particle Dynamics: Phase Transitions, Patterns and Synchronization) of the European Research Council Executive Agency (ERC) under the European Union's Horizon 2020 research and innovation programme (grant agreement No. 883363). The authors would like to thank Stefan Steinerberger for helpful comments on a previous draft.

\section{Appendix: Proof of Lemma \ref{lem_iso}}

\begin{lemma}[Loomis-Whitney inequality \cite{LW}]\label{lem_LW}
	Let $S\subseteq \mathbb{T}^d$ with $d>1$. Then 
	\begin{equation}
		|S|^{d-1} \le \prod_{j=1}^d |\pi_j(S)|,
	\end{equation}
	where $\pi_j$ is the projection onto the $j$-th coordinate hyperplane.
\end{lemma}

\begin{proof}[Proof of Lemma \ref{lem_iso}]
	The 1D case is trivial. In the rest of this proof, we will assume $d\ge 2$.
	
	Denote $\epsilon=|S_r\backslash S|$. 
	It is clear that we may assume 
	$r< \frac{0.1}{2(d+0.1)}$
	without loss of generality. 
	First notice that either $|S|\le 1/2$ or $|S^c|\le 1/2$. We will denote the restriction of $S$ onto lines of coordinate directions as $S|_{1;(x_2,\dots,x_d)}:=\{x_1\in \mathbb{T}:(x_1,x_2,\dots,x_d)\in S\}$ and similarly define $S|_{j;(x_1,\dots,\hat{x}_j,\dots,x_d)}$ with $\hat{\bx}_j=(x_1,\dots,\hat{x}_j,\dots,x_d)\in \mathbb{T}^{d-1}$.
	
	{\bf Case 1}: If $|S|\le 1/2$, then define
	\begin{equation}
		A_j = \{\hat{\bx}_j: S|_{j;\hat{\bx}_j} = \emptyset\},\quad B_j = \Big\{\hat{\bx}_j: \Big|S|_{j;\hat{\bx}_j}\Big| > \frac{d}{d+0.1}\Big\},\quad D_j = \mathbb{T}^{d-1}\backslash (A_j\cup B_j),\quad j=1,\dots,d.
	\end{equation}
	It is clear that $|B_1| \le (1+0.1/d)|S|$. For any $\hat{\bx}_1\in D_1$, we have $|(S|_{1;\hat{\bx}_1})^c| \ge \frac{0.1}{d+0.1} > 2r$, and thus after an $r$-expansion,
	\begin{equation}
		\Big|S_r|_{1;\hat{\bx}_1}\backslash S|_{1;\hat{\bx}_1}\Big| \ge 2r.
	\end{equation}
	Integrating in $\hat{\bx}_1\in D_1$, we see that
	\begin{equation}
		2r |D_1| \le \int_{D_1}\Big|S_r|_{1;\hat{\bx}_1}\backslash S|_{1;\hat{\bx}_1}\Big| \rd{\hat{\bx}_1} \le \int_{\mathbb{T}^{d-1}}\Big|S_r|_{1;\hat{\bx}_1}\backslash S|_{1;\hat{\bx}_1}\Big| \rd{\hat{\bx}_1} = \epsilon,
	\end{equation}
	i.e., $|D_1|\le \frac{\epsilon}{2r}$. Therefore, combined with the same estimates for other $B_j$ and $D_j$ and applying Lemma \ref{lem_LW}, we get
	\begin{equation}\label{Sdeq1}\begin{split}
			|S|^{d-1} \le & \prod_{j=1}^d|B_j\cup D_j| \le \Big((1+0.1/d)|S|+\frac{\epsilon}{2r}\Big)^d .
	\end{split}\end{equation}
	If $|S|>\frac{d}{0.1}\cdot\frac{\epsilon}{2r}$ were true, then we would have $|S|^{d-1} \le (1+0.2/d)^d|S|^d$, i.e., $1\le (1+0.2/d)^d |S|$, contradicting the assumption $|S|\le 1/2$ since $(1+0.2/d)^d \le e^{0.2} < 2$. Therefore we have $|S|\le \frac{d}{0.1}\cdot\frac{\epsilon}{2r}$. Substituting into the RHS of \eqref{Sdeq1}, we obtain $|S|^{d-1}\le C\frac{\epsilon^d}{r^d}$ which is the conclusion.
	
	{\bf Case 2}: If $|S^c|\le 1/2$, then define
	\begin{equation}
		A_j = \{\hat{\bx}_j: S|_{j;\hat{\bx}_j} = \emptyset\},\quad B_j = \Big\{\hat{\bx}_j: \Big|S^c|_{j;\hat{\bx}_j}\Big| \le 2r\Big\},\quad D_j = \mathbb{T}^{d-1}\backslash (A_j\cup B_j),\quad j=1,\dots,d.
	\end{equation}
	It is clear that $|A_1| \le |S^c|$. For any $\hat{\bx}_1\in B_1$, it is clear that $S_r^c|_{1;\hat{\bx}_1}=\emptyset$. Therefore 
	\begin{equation}
		|S^c\cap ( \mathbb{T}_{x_1}\times B_1) | \le \epsilon
	\end{equation}
	where $\mathbb{T}_{x_1}\times B_1$ denotes the set of points with $(x_2,\dots,x_d)$ coordinate in $B_1$ and $x_1$ coordinate in $\mathbb{T}$. For any $\hat{\bx}_1\in D_1$, we have $\Big|S_r|_{1;\hat{\bx}_1}\backslash S|_{1;\hat{\bx}_1}\Big| \ge 2r$, and thus we get $|D_1|\le \frac{\epsilon}{2r}$ as before. Therefore, combined with the same estimates for other $A_j$, $B_j$ and $D_j$,
	\begin{equation}\label{Scest}\begin{split}
			|S^c| \le & \Big|\{\bx\in S^c: (x_1,\dots,\hat{x}_j,\dots,x_d)\in A_j\cup D_j,\,j=1,\dots,d\}\Big| + \bigcup_{j=1}^d |S^c\cap (\mathbb{T}_{x_j}\times B_j )|\\
			\le & \Big(|S^c|+\frac{\epsilon}{r}\Big)^{d/(d-1)} + d\epsilon,
	\end{split}\end{equation}
	where in the second inequality we applied Lemma \ref{lem_LW} to the first term. 
	
	Then notice that $\epsilon = |S_r\backslash S| \ge c_1 r^d$ for some $c_1>0$. In fact, since $S_r^c\ne\emptyset$, we have $S_{\frac{2}{3}r}\backslash S_{\frac{1}{3}r} \ne\emptyset$. Take a point $\bx$ in this set, then it is clear that $B(\bx;\frac{1}{3}r)\subseteq S_r\backslash S$, which gives $|S_r\backslash S| \ge c_1 r^d$ with $c_1 = |B(0;1/3)|$. Therefore \eqref{Scest} gives
	\begin{equation}\begin{split}
			|S^c| \le  \Big(|S^c|+\frac{\epsilon}{r}\Big)^{d/(d-1)} + C_0\Big(\frac{\epsilon}{r}\Big)^{d/(d-1)},\quad C_0 = d c_1^{-1/(d-1)}.
	\end{split}\end{equation}
	
	Finally we prove $|S^c| \le C_1 (\frac{\epsilon}{r})^{d/(d-1)}$ for some $C_1>0$ to be determined.
	
	Suppose the contrary that $|S^c| > C_1 (\frac{\epsilon}{r})^{d/(d-1)}$. Let $C_2>0$ be a constant to be determined.
	\begin{itemize}
		\item If $C_1 (\frac{\epsilon}{r})^{d/(d-1)} \le |S^c| \le C_2\frac{\epsilon}{r}$, then
		\begin{equation}
			\Big(|S^c|+\frac{\epsilon}{r}\Big)^{d/(d-1)} + C_0\Big(\frac{\epsilon}{r}\Big)^{d/(d-1)} \le \Big((1+C_2)^{d/(d-1)}+C_0\Big)\Big(\frac{\epsilon}{r}\Big)^{d/(d-1)}.
		\end{equation}
		giving a contradiction if
		\begin{equation}\label{C12_1}
			(1+C_2)^{d/(d-1)}+C_0 < C_1.
		\end{equation}
		\item If $C_2\frac{\epsilon}{r} < |S^c| \le 1/2$, then
		\begin{equation}\begin{split}
				\Big(|S^c|+\frac{\epsilon}{r}\Big)^{d/(d-1)} + C_0\Big(\frac{\epsilon}{r}\Big)^{d/(d-1)}\le & \Big((1+\frac{1}{C_2})^{d/(d-1)}+\frac{C_0}{C_2^{d/(d-1)}}\Big)|S^c|^{d/(d-1)} \\
				\le & 2^{-1/(d-1)}\Big((1+\frac{1}{C_2})^{d/(d-1)}+\frac{C_0}{C_2^{d/(d-1)}}\Big)|S^c|. \\
		\end{split}\end{equation}
		giving a contradiction if
		\begin{equation}\label{C12_2}
			\Big(1+\frac{1}{C_2}\Big)^{d/(d-1)}+\frac{C_0}{C_2^{d/(d-1)}}< 2^{1/(d-1)} .
		\end{equation}
	\end{itemize} 
	To determine the choice of $C_1,C_2$, we first choose $C_2$ large enough so that \eqref{C12_2} is satisfied. Then we choose $C_1$ large enough so that \eqref{C12_1} is satisfied. This finishes the proof of this lemma.
	
\end{proof}

\section{Appendix: Proof of Lemma \ref{lem_layer}}\label{sec_reg}

The following lemma is straightforward.
\begin{lemma}\label{lem_Reg0}
	For any $S\subseteq \mathbb{T}^d$ and $r>0$,
	\begin{equation}\label{Reg}
		\bx\in \Reg_r(S) \Leftrightarrow B(\bx;r)\subseteq S_r.
	\end{equation}
	It follows that $S\subseteq \Reg_r(S)$. Furthermore, $S_r = (\Reg_r(S))_r$.
\end{lemma}
This lemma motivates the following definition.
\begin{definition}
	Let $r>0$. A set $S\subseteq \mathbb{T}^d$ is \emph{$r$-regular} if $S=\Reg_r(S)$.
\end{definition}

It is clear from Lemma \ref{lem_Reg0} that $\Reg_r(S)$ is $r$-regular. Next we give some basic properties of $r$-regular sets.
\begin{lemma}\label{lem_Reg}
	Let $S$ be an $r$-regular set. Then
	\begin{enumerate}
		\item $S$ is closed.
		\item For any $\bx\in\partial S$, there exists $\by\in S_r^c$ such that $|\by-\bx|=r$ and $B(\by;r)\cap S=\emptyset$.
		\item For any $0<r'<r$, $S$ is $r'$-regular.
	\end{enumerate}
\end{lemma}

\begin{proof}
	Item 1: It is clear that $(S_r^c)_r$ is open, and thus $S=\Reg_r(S)=(S_r^c)_r^c$ is closed.
	
	Item 2: Take $\bx\in\partial S\subseteq S$. We claim that $\dist(\bx,S_r^c)=r$. It is clear that $\dist(\bx,S_r^c)\ge r$. Suppose $\dist(\bx,S_r^c)> r+\epsilon$ for some $\epsilon>0$, then $B(\bx;r+\epsilon)\subseteq S_r$, and then $(S\cup B(\bx;\epsilon))_r=S_r$. Therefore $S\cup B(\bx;\epsilon)\subseteq \Reg_r(S\cup B(\bx_0;\epsilon)) = \Reg_r(S) = S$, contradicting the assumption that $\bx\in \partial S$. Therefore we see that $\dist(\bx,S_r^c)=r$.
	
	Since $S_r^c$ is closed, $\dist(\bx,S_r^c)=r$ is achieved at some $\by\in S_r^c$, i.e., $|\by-\bx|=r$. We claim that $B(\by;r)\cap S=\emptyset$. Suppose not, then there exists some $\bz\in S$ with $|\by-\bz|<r$. This contradicts $\by\in S_r^c$. Therefore we get $B(\by;r)\cap S=\emptyset$.
	
	Item 3: Let $\bx\in \Reg_{r'}(S)$. Then by \eqref{Reg}, $B(\bx;r')\subseteq S_{r'}$. Then $B(\bx;r)=(B(\bx;r'))_{r-r'} \subseteq (S_{r'})_{r-r'}=S_r$. Therefore $\bx\in \Reg_r(S)=S$ since $S$ is $r$-regular. Therefore $\Reg_{r'}(S)\subseteq S$, i.e., $S$ is $r'$-regular.
	
	
\end{proof}

\begin{proof}[Proof of Lemma \ref{lem_layer}]
	We may assume that $S$ is $r$-regular because $r$-regularizing $S$ would make $S$ larger with $S_r,S_{2r}$ remaining the same.
	
	By the definition of $S_r$, it is straightforward to see
	\begin{equation}
		\widebar{S_{0.9r}\backslash S}\subseteq \bigcup_{\bx\in\partial S} B(\bx;r).
	\end{equation}
	Since $ \widebar{S_{0.9r}\backslash S}$ is compact, one can apply Vitali covering lemma to get a finite collection $B(\bx_j;r)$ for $j=1,\dots,n$ such that $\{B(\bx_j;r/3)\}$ is disjoint, and  $\widebar{S_{0.9r}\backslash S}\subseteq \bigcup_{j=1}^n B(\bx_j;r)$. For every $j$, we apply items 2 and 3 of Lemma \ref{lem_Reg} to see that there exists $\by_j$ such that $|\bx_j-\by_j|=r/3$ and $B(\by_j;r/3)\cap S=\emptyset$. Since $B(\bx_j;r/3)\subseteq S_r$, we see that
	\begin{equation}
		|B(\bx_j;r/3)\cap (S_r\backslash S)| = |B(\bx_j;r/3)\backslash S| \ge |B(\bx_j;r/3)\cap B(\by_j;r/3)| = c r^d,
	\end{equation}
	where $c$ only depends on $d$. Since $\{B(\bx_j;r/3)\}$ is disjoint, we get
	\begin{equation}
		|S_r\backslash S| \ge \sum_{j=1}^n |B(\bx_j;r/3)\cap (S_r\backslash S)| \ge c n r^d.
	\end{equation}
	
	On the other hand, since $\widebar{S_{0.9r}\backslash S}\subseteq \bigcup_{j=1}^n B(\bx_j;r)$, we see that
	\begin{equation}
		S_{2r}\backslash S\subseteq \bigcup_{j=1}^n B(\bx_j;2.1r).
	\end{equation}
	Therefore
	\begin{equation}
		|S_{2r}\backslash S_r| \le |S_{2r}\backslash S| \le \sum_{j=1}^n|B(\bx_j;2.1r)| = C n r^d,
	\end{equation}
	which finishes the proof.
\end{proof}

\section{Appendix: Explicit formula for $W_s$}
We give an explicit formula for the periodized Riesz potential $W_s$ for $0<s<d$. We will show that the formula (7) in \cite{HSS14} with $\mu(t) = t^{s/2-1}$ indeed gives $W_s$, and obtain the regularity of $W_s$ accordingly.

\begin{lemma}\label{lem_period}
	$W_s$ in \eqref{W} satisfies the following regularity conditions:
	\begin{enumerate}
		\item For any $s<d$, $W_s$ is smooth on $\mathbb{T}^d\backslash \{0\}$.
		\item For $0<s<d$, $W_s(\bx) - c_s|\bx|^{-s}$ is a smooth function near 0, for some $c_s>0$.
		\item For $s=0$, $W_s\in L^p$ for any $1\le p < \infty$, and $\lim_{\bx\rightarrow 0} W_s(\bx)=\infty$.
		\item For $s<0$, $W_s$ is continuous.
	\end{enumerate}
	Furthermore, $W_s$ satisfies ({\bf H1})-({\bf H3}) for any $s<d$.
\end{lemma}

\begin{proof}
	
	First, item 4 is clear since $\hat{W}_s \in \ell^1$ for $s<0$.
	
	Then, assuming $0<s<d$, we claim that 
	\begin{equation}\label{lem_period_1}
		cW_s(\bx) + C = \sum_{\bj\in\mathbb{Z}^d} \int_1^\infty e^{-|\bx-\bj|^2 t}t^{s/2-1}\rd{t} + \sum_{\bk\in\mathbb{Z}^d\backslash \{0\}} e^{2\pi i \bk\cdot \bx} \int_0^1 \pi^{d/2}e^{-\pi^2 |\bk|^2/t}t^{s/2-1-d/2}\rd{t}.
	\end{equation}
	
	It is clear that the RHS of \eqref{lem_period_1} is well-defined on $\mathbb{T}^d$ and finite at every $\bx\in\mathbb{T}^d\backslash \{0\}$. Denoting this two terms as $w_1(\bx)$ and $w_2(\bx)$. We first calculate its Fourier coefficients. It is clear that $\hat{w}_2(\bk)=\int_0^1 \pi^{d/2}e^{-\pi^2 |\bk|^2/t}t^{s/2-1-d/2}\rd{t}$ for any $\bk\in\mathbb{Z}^d\backslash \{0\}$. For $w_1$,
	\begin{equation}\begin{split}
			\hat{w}_1(\bk) = & \int_{\mathbb{R}^d}e^{2\pi i \bk\cdot \bx} \int_1^\infty e^{-|\bx|^2 t}t^{s/2-1}\rd{t} \rd{\bx}=  \int_1^\infty \int_{\mathbb{R}^d}e^{2\pi i \bk\cdot \bx}e^{-|\bx|^2 t} \rd{\bx}\, t^{s/2-1}\rd{t} \\
			= & \int_1^\infty \pi^{d/2}e^{-\pi^2 |\bk|^2/t} t^{s/2-1-d/2}\rd{t}. \\
	\end{split}\end{equation}
	Therefore the $\bk$-th Fourier coefficient of the RHS of \eqref{lem_period_1} is
	\begin{equation}
		\int_0^\infty \pi^{d/2}e^{-\pi^2 |\bk|^2/t} t^{s/2-1-d/2}\rd{t} =\int_0^\infty \pi^{d/2}e^{-\pi^2 |\bk|^2t} t^{-s/2-1+d/2}\rd{t} = c|\bk|^{-d+s},
	\end{equation}
	by a change of variable $t_1=1/t$, and applying the formula $y^{-s/2} = \frac{1}{\Gamma(s/2)}\int_0^\infty t^{s/2-1}e^{-ty}\rd{t}$ with $y=\pi^2|\bk|^2$ and $s$ replaced by $(d-s)$. This verifies the equality in \eqref{lem_period_1} in view of \eqref{W}.
	
	It is clear that $w_2$ is a smooth function since its Fourier coefficients have fast decay at infinity. For $w_1$, one can extract the $\bj=0$ term, which is equal to $c|\bx|^{-s}-\int_0^1 e^{-|\bx|^2 t}t^{s/2-1}\rd{t} $, and write
	\begin{equation}
		w_1(\bx) = c|\bx|^{-s} - \int_0^1 e^{-|\bx|^2 t}t^{s/2-1}\rd{t}+ \sum_{\bj\in\mathbb{Z}^d\backslash \{0\}} \int_1^\infty e^{-|\bx-\bj|^2 t}t^{s/2-1}\rd{t} ,
	\end{equation}
	for $\bx\in [-1/2,1/2)^d$. The term $\int_0^1 e^{-|\bx|^2 t}t^{s/2-1}\rd{t}$ and the last summation are clearly smooth on $[-1/2,1/2)^d$ since $|\bx-\bj|$ is away from 0. Therefore we see that $W_s$ is smooth on $\mathbb{T}^d\backslash \{0\}$ and $W_s(\bx) - c|\bx|^{-s}$ is a smooth function near 0. This proves item 2, as well as the case $0<s<d$ for item 1.
	
	For $s\le 0$, we take $n\in \mathbb{N}$, and notice that \eqref{W} implies
	\begin{equation}
		W_s = \underbrace{W_{s'}*\cdots*W_{s'}}_\text{n times},\quad s' = \frac{n-1}{n}d+\frac{1}{n}s .
	\end{equation}
	For $n$ sufficiently large, we have $0<s'<d$, and item 1 for $s$ follows from item 1 for $s'$ which we have proved. For item 3, we take $s=0$, $n=2$, $s'=\frac{d}{2}$, and notice that items 1, 2 for $s'$ implies $W_{s'}\in L^{2-\epsilon}$ for any $\epsilon>0$. Then it follows from Young's inequality that $W_s\in L^p$ for any $1\le p < \infty$, and the singularity structure given in item 2 for $s'$ shows that $\lim_{\bx\rightarrow 0} W_s(\bx)=\infty$.
	
	By definition, $W_s$ always satisfies ({\bf H2}), and ({\bf H1}) is also clear from items 1-4. For ({\bf H3}), the case $0<s<d$ follows from item 2 and the same local property of the power-law potential $|\bx|^{-s}$. The case $s<0$ follows from item 4. For the case $s=0$, notice that $W_s=W_{s'}*W_{s'}$ with $s'=\frac{d}{2}$, and $W_{s'}$ satisfies ({\bf H3}). Then we see that for sufficiently large $C_1>0$,
	\begin{equation}\begin{split}
			\frac{1}{|B(0;r)|} & \int_{B(\bx;r)}(C_1^2+W_s(\by))\rd{\by} = \frac{1}{|B(0;r)|}((C_1^2+W_s)*\chi_{B(0;r)})(\bx)\\ 
			= & \frac{1}{|B(0;r)|}\Big((C_1+W_{s'})*\big((C_1+W_{s'})*\chi_{B(0;r)}\big)\Big)(\bx) 
			\le C((C_1+W_{s'})*(C_1+W_{s'}))(\bx) \\
			= & C(C_1^2+W_s(\bx)).
	\end{split}\end{equation}
	using ({\bf H3}) and the mean-zero property of $W_{s'}$. This proves ({\bf H3}) for $W_s$.
\end{proof}

\bibliographystyle{alpha}
\bibliography{set1.bib}
\vspace{ 1cm}

\Addresses
\end{document}